\documentclass[reqno,12pt]{amsart}
 \usepackage{amsmath}
 \usepackage{amsthm}
 \usepackage{amssymb}
 \usepackage{latexsym,longtable}
 \usepackage{url}
 \usepackage{graphicx}
 \usepackage{multicol}
 \usepackage{mathrsfs}
 \usepackage{young}
 \usepackage[vcentermath,enableskew,stdtext]{youngtab}
 \usepackage[left=1.07in,right=1.07in,top=1.1in,bottom=1.14in]{geometry}
 \usepackage[all]{xy}
    \SelectTips{cm}{10}     %use the nicer arrowheads
    \everyxy={<2.5em,0em>:} % Sets the scale I like
 \usepackage{fancyhdr}      %%%%%%%%%%%% Pagestyle stuff %%%%%%%%%%%%%%%
 \usepackage{multicol}
 \usepackage{multirow}
 \usepackage{enumitem}
 \usepackage{stmaryrd}
 \usepackage{etex}
 \usepackage{tikz}
 \usepackage{float}
 \usepackage{array}
 \usepackage{colortbl}
 \usepackage{mathtools}
 \usepackage{bm}
 \restylefloat{figure}
 \usetikzlibrary{shapes}
 \usetikzlibrary{arrows}
 \usetikzlibrary{positioning}
 \usetikzlibrary{decorations.pathmorphing}
 \usetikzlibrary{decorations.markings}
 \usepackage{color}
\definecolor{gray}{rgb}{.5, .5, .5}

\usepackage{amsfonts,epsfig,color}
\makeatletter
\newcommand*{\centerfloat}{%
  \parindent \z@
  \leftskip \z@ \@plus 1fil \@minus \textwidth
  \rightskip\leftskip
  \parfillskip \z@skip}
\makeatother
\newcounter{ctr}
\theoremstyle{plain}
\newtheorem{theorem}{Theorem}[section]
\newtheorem{lemma}[theorem]{Lemma}
\newtheorem{corollary}[theorem]{Corollary}
\newtheorem{proposition}[theorem]{Proposition}
\newtheorem{conjecture}[theorem]{Conjecture}

\newtheorem{propdef}[theorem]{Proposition-Definition}

\theoremstyle{definition}
\newtheorem{definition}[theorem]{Definition}

\newtheorem{remark}[theorem]{Remark}
\newtheorem{example}[theorem]{Example}

\newcommand{\ignore}[1]{}

\newcommand\drawcell[2]{
\draw (0+#1,0+#2)--(1+#1,0+#2)--(1+#1,1+#2)--(0+#1,1+#2)--(0+#1,0+#2);
}

\renewcommand{\a}{\ensuremath{\mathfrak{a}}}

\newcommand{\G}{\ensuremath{\mathcal{G}}}

\newcommand{\QQ}{\ensuremath{\mathbb{Q}}}
\newcommand{\R}{\ensuremath{\mathscr{R}}}

\newcommand{\sgn}{\text{\rm sgn}}

\newcommand{\U}{\mathcal{U}}

\newcommand{\ZZ}{\ensuremath{\mathbb{Z}}}
\newcommand{\be}{\begin{equation}}
\newcommand{\ee}{\end{equation}}

\renewcommand{\S}{\ensuremath{\mathcal{S}}}
\newcommand{\tsr}{\ensuremath{\otimes}}

\newcommand{\rsd}[1]{\ensuremath{\hat{#1}}}
 
\newcommand{\sh}{\text{\rm sh}}

\def\Tiny{\fontsize{6pt}{6pt}\selectfont}

\newcommand{\crc}[1]{\ensuremath{\overline{#1}}\vphantom{\underline{\overline{#1}}}}
\newcommand{\stand}{\ensuremath{\text{\rm st}}} %??? or stand

\newcommand\mybox[1]{
\vcenter{
\let\\=\cr
\baselineskip=-16000pt \lineskiplimit=16000pt \lineskip=0pt
\halign{&\boxcell{##}\cr\vline#1\vline\crcr}}}

\newcommand{\boxcell}[1]{{%
\unitlength=\cellsizeCol
\begin{picture}(1,1)
\put(0,0){\makebox(1,1){$#1$}}
\put(0,0){\line(1,0){1}}
\put(0,1){\line(1,0){1}}
\end{picture}%
}}

\newcommand\pad[1]{
\vtop{
\let\\=\cr
\baselineskip=-16000pt
\lineskiplimit=16000pt
\lineskip=0pt
\halign{& \inviscell{##} \cr #1 \crcr} }
\hspace{-.73ex}}

\newcommand{\inviscell}[1]{{%
\unitlength=\cellsizeCol
\begin{picture}(1,1)
\put(0,0){\makebox(1,1){$#1$}}
\end{picture}%
}}

\newlength{\cellsize}
\newcommand\tableau[1]{
\vcenter{
\let\\=\cr
\baselineskip=-16000pt \lineskiplimit=16000pt \lineskip=0pt
\halign{&\tableaucell{##}\cr#1\crcr}}}

\newcommand{\tableaucell}[1]{{%
\def \arg{#1}\def \void{}%
\ifx \void \arg
\vbox to \cellsize{\vfil \hrule width \cellsize height 0pt}%
\else \unitlength=\cellsize
\begin{picture}(1,1)
\put(0,0){\makebox(1,1){$#1\vphantom{\crc{#1}}$}}
\put(0,0){\line(1,0){1}}
\put(0,1){\line(1,0){1}}
\put(0,0){\line(0,1){1}}
\put(1,0){\line(0,1){1}}
\end{picture}%
\fi}}

\newcommand\boldtableau[1]{
\vcenter{
\let\\=\cr
\baselineskip=-16000pt \lineskiplimit=16000pt \lineskip=0pt
\halign{&\boldtableaucell{##}\cr#1\crcr}}}

\newcommand{\boldtableaucell}[1]{{%
\def \arg{#1}\def \void{}%
\ifx \void \arg
\vbox to \cellsize{\vfil \hrule width \cellsize height 0pt}%
\else \unitlength=\cellsize
\begin{picture}(1,1)
\put(0,0){\makebox(1,1){$\mathbf{#1\vphantom{\crc{#1}}}$}}
\put(0,0){\line(1,0){1}}
\put(0,1){\line(1,0){1}}
\put(0,0){\line(0,1){1}}
\put(1,0){\line(0,1){1}}
\end{picture}%
\fi}}

\newlength{\colskip}
\newlength{\dwidth}

\newcommand{\partition}[1]{{\setlength{\cellsize}{1ex} \tiny \tableau{#1}}}

\newcommand{\Ilam}[1]{\ensuremath{{{I_{#1}^\text{\rm Lam}}}}}

\newcommand{\Irkst}[1]{\ensuremath{{{I_{#1}^{\text{\rm KR}, \, \stand}}}}}

\newcommand{\Iassafst}[1]{\ensuremath{{I_{#1}^\text{\rm Assaf, \stand}}}}

\newcommand{\corner}{\ensuremath{{\llcorner}}}
\newcommand{\spin}{\ensuremath{{{\text{\rm spin}}}}}

\newcommand{\inv}{\ensuremath{{\text{\rm inv}}}}
\newcommand{\Des}{\ensuremath{{\text{\rm Des}}}}
\newcommand{\invi}[1]{\ensuremath{{\text{\rm inv$_{#1}'$}}}}
\newcommand{\Desi}[1]{\ensuremath{{\text{\rm Des$_{#1}'$}}}}
\newcommand{\invip}[1]{\ensuremath{{\text{\rm inv$_{#1}''$}}}}
\newcommand{\Desip}[1]{\ensuremath{{\text{\rm Des$_{#1}''$}}}}
\newcommand{\WRib}{\ensuremath{{\text{\rm WRib}}}}
\newcommand{\Wi}[1]{\ensuremath{{\text{\rm W$_{#1}'$}}}}

\newcommand{\qlam}{\ensuremath{{\hat{q}}}}
\newcommand{\field}{\ensuremath{\mathbf{A}}}

\newcommand{\e}{\mathsf}
\newcommand{\sqread}{\text{\rm sqread}}
\newcommand{\quot}{\text{\rm quot}}
\newcommand{\core}{\text{\rm core}}

\newcommand{\SYT}{\text{\rm SYT}}

\newcommand{\RSST}{\text{\rm RSST}}
\newcommand{\Tab}{\text{\rm Tab}}

\newcommand{\nearr}{{\scalebox{.75}{$\nearrow$}}}

\newcommand{\searr}{{\scalebox{.75}{$\searrow$}}}
\newcommand{\nwarr}{{\scalebox{.75}{$\nwarrow$}}}
\newcommand{\nearrsub}{{\scalebox{.5}{$\nearrow$}}}

\newcommand{\searrsub}{{\scalebox{.5}{$\searrow$}}}

\newcommand{\Jnot}[1]{\ensuremath{{\hspace{1pt}\text{\rm:}~\raisebox{1pt}{$\e{#1}$}~\text{\rm:}\hspace{1pt}}}}
\newcommand{\Jnotb}[2]{\ensuremath{{{}_{#1}{}^{\e{#2}}}}}
\newcommand{\QQA}{{\ensuremath{\QQ[\qlam,\qlam^{-1}]}}}

\title{Haglund's conjecture on 3-column Macdonald polynomials}


\begin{document}

\author{Jonah Blasiak}
\email{jblasiak@gmail.com}
\address{Department of Mathematics, Drexel University, Philadelphia, PA 19104}
%\thanks{2010 \emph{Mathematics Subject Classification}.}
\thanks{This work was supported by NSF Grant DMS-14071174.}

%The proof uses Lam's algebra of ribbon Schur operators \cite{LamRibbon} and the setup of a joint paper with Fomin \cite{BF}.  The proof requires expressing a noncommutative Schur function as  a positive sum of monomials in Lam's algebra.
\begin{abstract}
We prove a positive combinatorial formula for the Schur expansion of LLT polynomials indexed by a 3-tuple of skew shapes.
This verifies a conjecture of Haglund \cite{Haglund}.
The proof requires expressing a noncommutative Schur function as  a positive sum of monomials in Lam's algebra of ribbon Schur operators \cite{LamRibbon}.
Combining this result with the expression of Haglund, Haiman, and Loehr \cite{HHL} for transformed Macdonald polynomials in terms of LLT polynomials then yields a positive combinatorial rule for transformed Macdonald polynomials indexed by a shape with 3 columns.
\end{abstract}
%arXiv version
%We prove a positive combinatorial formula for the Schur expansion of LLT polynomials indexed by a 3-tuple of skew shapes.
%This verifies a conjecture of Haglund.
%The proof requires expressing a noncommutative Schur function as  a positive sum of monomials in Lam's algebra of ribbon Schur operators.
%Combining this result with the expression of Haglund, Haiman, and Loehr for transformed Macdonald polynomials in terms of LLT polynomials then yields a positive combinatorial rule for transformed Macdonald polynomials indexed by a shape with 3 columns.

%Subjects:	Combinatorics (math.CO); math.RA - Rings and Algebras
%MSC classes:	05E05  	Symmetric functions and generalizations
\maketitle

\section{Introduction}
%The LLT polynomials $G^{(k)}_{\theta}(\mathbf{x};q)$ are the $q$-generating functions over  $k$-ribbon tableaux weighted by spin introduced by Lascoux, Leclerc, and Thibon.
%The LLT polynomials $G^{(k)}_{\theta}(\mathbf{x};q)$ are  $q$-symmetric functions, defined combinatorially as generating functions over  $k$-ribbon tableaux weighted by spin.
%Introduced by Lascoux, Leclerc, and Thibon in 1997 \cite{LLT}, they have since
%as generating functions over  $k$-ribbon tableaux of skew shape  $\theta$ weighted by spin.
In the late 90's,  Lascoux, Leclerc, and Thibon \cite{LLT} defined a family of symmetric functions depending on a parameter $q$, in terms of ribbon tableaux and the spin statistic.
These symmetric functions, known as \emph{LLT polynomials},
are now fundamental in the study of Macdonald polynomials and diagonal coinvariants, and have intriguing connections to Kazhdan-Lusztig theory, $k$-Schur functions, and plethysm.
%For instance,  the famous result of Haglund, Haiman, and Loehr \cite{HHL} gives a  positive combinatorial formula for the transformed Macdonald polynomials $\tilde{H}_\mu(\mathbf{x};q,t)$ in terms of monomials, and can  be interpreted as an  expression for $\tilde{H}_\mu(\mathbf{x};q,t)$ as a positive sum of the LLT polynomials $G^{(k)}_{\theta}(\mathbf{x};q)$ for  $k$ the number of columns of $\mu$.
%A famous result of Haglund, Haiman, and Loehr \cite{HHL} can be interpreted as a formula for the transformed Macdonald polynomials $\tilde{H}_\mu(\mathbf{x};q,t)$ as a positive sum of the LLT polynomials $G^{(k)}_{\theta}(\mathbf{x};q)$ for $k$ the number of columns of $\mu$.
%%indexed by a $k$-tuple of ribbon skew shapes (where $k=\mu_1$).
%Haglund, Haiman, Loehr, Remmel, and Ulyanov conjecture \cite{HHLRU} a positive expression for the graded character of the diagonal coinvariants in terms of LLT polynomials.

In this paper we work with the version of LLT polynomials from \cite{HHLRU}, which are indexed by $k$-tuples of skew shapes.
We give a positive combinatorial formula for the Schur expansion of LLT polynomials indexed by a 3-tuple of skew shapes.
The Haglund-Haiman-Loehr formula \cite{HHL} expresses the transformed Macdonald polynomials $\tilde{H}_\mu(\mathbf{x};q,t)$ as a positive sum of the LLT polynomials
indexed by a $k$-tuple of ribbon shapes, where $k$ is the number of columns of $\mu$.
Haglund \cite{Haglund} conjectured a formula for the LLT polynomials that appear in this formula in the case that the partition  $\mu$ has 3 columns.
Our result proves and generalizes this formula.

%The \emph{$q$-Littlewood-Richardson coefficients}  $c_{\theta, \lambda}^{(k)}(q)$ are the coefficients of the Schur expansion of LLT polynomials, i.e.
The \emph{new variant $q$-Littlewood-Richardson coefficients} $\mathfrak{c}_{\bm{\beta}}^\lambda(q)$ are the coefficients in the Schur expansion of
LLT polynomials, i.e.
\[\mathcal{G}_{\bm{\beta}}(\mathbf{x};q) = \sum_{\lambda}\mathfrak{c}_{\bm{\beta}}^\lambda(q)s_\lambda(\mathbf{x}),\]
where $\mathcal{G}_{\bm{\beta}}(\mathbf{x};q)$ is the LLT polynomial indexed by a  $k$-tuple  $\bm{\beta}$ of skew shapes; the adjective new variant
indicates that these correspond to the version of LLT polynomials from \cite{HHLRU}, not the version used in \cite{LLT, LT00, LamRibbon}.
These coefficients are polynomials in $q$ with nonnegative integer coefficients.
%This was conjectured in \cite{LLT} and proved in the case
%the coefficient of  $s_\lambda(\mathbf{x})$ in $\mathcal{G}_{\bm{\beta}}(\mathbf{x};q)$
This was proved in the case that $\bm{\beta}$ is a tuple of partition shapes \cite{LT00} by showing that these coefficients are essentially  parabolic Kazhdan-Lusztig polynomials.
The general case was proved in \cite{GH}, also using Kazhdan-Lusztig theory.
The paper \cite{Sami} claims a combinatorial proof of positivity.
However, neither of these methods yields an explicit positive combinatorial interpretation of the new variant $q$-Littlewood-Richardson coefficients for  $k > 2$.
(The approach of \cite{Sami}, though combinatorial, involves an intricate algorithm to transform a  D graph into a dual equivalence graph and has yet to produce explicit formulas for  $k > 2$.)

Explicit combinatorial formulas for the new variant $q$-Littlewood-Richardson coefficients have been found in the following cases.
% for $c_{\theta, \lambda}^{(k)}(q)$:
A combinatorial interpretation
for the $k=2$ case was stated by Carr\'{e} and Leclerc in \cite{CL95}, and its proof was completed by van Leeuwen in \cite{vL00} (see \cite[\textsection9]{HHL}).
%?? discussion about this in LLT is slim compared to HHL
Assaf \cite{Ass08} gave another interpretation of these coefficients.
%Assaf gave two interpretations \cite{Ass08,Sami} of these coefficients;
%the second comes from observing that the  $D$-graphs arising in the  $k=2$ case are actually dual equivalence graphs. %(Those for  $k>2$ are not dual equivalence graphs in general, which gives some measure of the difficulties encountered in  $k>2$.)
Roberts \cite{RobertsDgraph} extended the work of Assaf to give an explicit formula for $\mathfrak{c}_{\bm{\beta}}^\lambda(q)$ in the case that the diameter of  $\bm{\beta}$ is  $\le 3$,
where the diameter of a $k$-tuple $\bm{\beta}$ of skew shapes is
\begin{align}\label{e diameter}
\max\big\{ \big|C \cap \{i, i+1, \dots, i+k\}\big| : i \in \ZZ \big\},
\end{align}
where $C$ is the set of distinct shifted contents of the cells of $\bm{\beta}$ (see \eqref{e intro tilde c} below).
Formulas for the coefficient of  $s_\lambda(\mathbf{x})$ in $\tilde{H}_\mu(\mathbf{x};q,t)$ are known when  $\lambda$ or  $\mu$ is a hook shape and
when  $\mu$ has two rows or two columns.
%and in certain cases obtained by adding one or two cells to these shapes.
Fishel \cite{Fis95} gave the first combinatorial interpretation for such coefficients in the case  $\mu$ has 2 columns using rigged configurations.
Zabrocki and Lapointe-Morse also gave formulas for this case \cite{Zab99,LM03}.
%\cite{HHL} also but this is the same as Carre Leclerc? yes

Let  $\U$ be the free associative algebra in the noncommuting variables $u_i$,  $i \in \ZZ$.
The \emph{plactic algebra} is the quotient of $\U$ by the Knuth equivalence relations.
It has been known since the work of Lascoux and Sch\"utzenberger \cite{LS} that the plactic algebra contains a subalgebra isomorphic to the ring of symmetric functions, equipped with a basis of noncommutative versions of Schur functions.
Fomin and Greene \cite{FG} showed that a similar story holds if certain pairs of Knuth relations are replaced by weaker four-term relations.
This yields positive formulae for the Schur expansions of a large class of symmetric functions that includes the Stanley symmetric functions and stable Grothendieck polynomials.
Lam \cite{LamRibbon} realized later  that some of this machinery can be applied to LLT polynomials.
To do this, he defined the \emph{algebra of ribbon Schur operators} $\U/\Ilam{k}$ to be the algebra generated by operators  $u_i$ which act on partitions by adding  $k$-ribbons.
%(see \textsection\ref{ss Lams algebra} for details).
He also explicitly described the relations satisfied by the  $u_i$, which we take here as the definition of this algebra (see \textsection\ref{ss Lams algebra}).
%\[\nu \cdot u_i = \begin{cases}
%q^{\spin(\mu/\nu)}\mu \qquad & \text{if $\mu/\nu$ is a $k$-ribbon of content $i$,} \\
%0 \qquad & \text{otherwise,}
%\end{cases}\]
%where the content of a ribbon is defined to be the maximum of the contents of its cells.
%adding a $k$-ribbon  $R$ of content  $i$ and multiplying by $q^{\spin(R)}$ if .
%, which is a quotient of the
Lam gave a simple interpretation of LLT polynomials using this algebra (Remark \ref{r spin LLT}).

The main difficulty in carrying out the Fomin-Greene approach in this setting is
expressing the noncommutative Schur functions  $\mathfrak{J}_\lambda(\mathbf{u})$ (see \textsection\ref{ss noncommutative flagged schur}) as a positive sum of monomials in  $\U/\Ilam{k}$.
Lam does this for $\lambda$ of the form  $(a,1^b)$,  $(a,2)$, $(2,2,1^a)$ and $k$ arbitrary.

Our main theorem is
\begin{theorem}
In the algebra  $\U/\Ilam{3}$, the noncommutative Schur function $\mathfrak{J}_\lambda(\mathbf{u})$ is equal to the following positive sum of monomials
\label{t J intro}
\[
\mathfrak{J}_\lambda(\mathbf{u}) = \sum_{T \in \RSST_\lambda} \sqread(T).
%\qquad \ \ \ \ \ \mathfrak{J}_\lambda(\mathbf{u}) = \sum_{T \in \RSST_\lambda} \sqread(T) \qquad \ \text{in }\, \U/\Ilam{3}.
\]
\end{theorem}
Here, $\RSST_\lambda$ is the set  of semistandard Young tableaux of shape  $\lambda$ that are row strict and such that entries increase in increments of at least 3
along diagonals. To define $\sqread(T)$, first
draw arrows as shown between entries of $T$ for each of its $2\times 2$ subtableaux of the following forms:
%containing four consecutive letters:
\setlength{\cellsize}{3.4ex}
\[{\tiny
\tableau{{\color{black}\put(.8,-1.6){\vector(-1,1){.6}}}\put(-.1,-0.96){$a$}&a\text{+}1\\a\text{+}2&a\text{+}3}\qquad
\tableau{{\color{black}\put(.8,-1.6){\vector(-1,1){.6}}}\put(-.1,-0.96){$a$}&a\text{+}2\\a\text{+}2&a\text{+}3}\qquad\qquad
\tableau{{\color{black}\put(.25,-1.1){\vector(1,-1){.6}}}\put(-.1,-1){$a$}&a\text{+}2\\a\text{+}1&a\text{+}3}\qquad
\tableau{{\color{black}\put(.25,-1.1){\vector(1,-1){.6}}}\put(-.1,-1){$a$}&a\text{+}1\\a\text{+}1&a\text{+}3}}.\]
Then  $\sqread(T)$ is a reading word of $T$  in which the
tail of each arrow appears before the head of each arrow
(any reading word satisfying these properties can be used in Theorem \ref{t J intro} and Corollary \ref{c main intro} below---$\sqread(T)$ is just a convenient choice).
For example,
\setlength{\cellsize}{1.9ex}
\[
\sqread\bigg(\,
{ \tiny \tableau{
1&2&4&6\\3&4&5&7\\8\\}}
\, \bigg) = \e{834152476}.
\]
\setlength{\cellsize}{2.1ex}

We show that by an adaptation of the Fomin-Greene theory of noncommutative Schur functions similar to Lam's \cite[Section 6]{LamRibbon} (see \textsection\ref{ss qLR coefs}),
the coefficient of a Schur function in an LLT polynomial is equal to $\langle \mathfrak{J}_\lambda(\mathbf{u}), f \rangle$, where  $f$ is a certain element of $\U$ that encodes an LLT polynomial and  $\langle \cdot, \cdot \rangle$ denotes the symmetric bilinear form on  $\U$ for which monomials form an orthonormal basis.
Hence Theorem \ref{t J intro} yields a positive combinatorial formula for the new variant $q$-Littlewood-Richardson coefficients
$\mathfrak{c}_{\bm{\beta}}^\lambda(q)$ for 3-tuples $\bm{\beta}$ of skew shapes.
We now state this formula in a special case which implies Haglund's conjecture and therefore, by \cite{HHL}, yields a formula for 3-column Macdonald polynomials
(see Corollary \ref{c main} for the full statement with no restriction on the 3-tuple of skew shapes, and see the discussion after Corollary \ref{c main} for the precise relation to Haglund's conjecture).

Let $\bm{\beta} =(\beta^{(0)},\dots,\beta^{(k-1)})$ be a $k$-tuple of skew shapes.
The \emph{shifted content} of a cell  $z$ of  $\bm{\beta}$ is
\begin{equation}
\label{e intro tilde c}
\tilde{c}(z) = k\cdot c(z)+i,
\end{equation}
when $z \in \beta^{(i)}$ and where $c(z)$ is the usual content of $z$ regarded as a cell of $\beta^{(i)}$.
For $\bm{\beta}$ such that each $\beta^{(i)}$ contains no $2 \times 2$ square, define $\Wi{k}(\bm{\beta})$ to be the set of words $\e{v}$ such that
\begin{itemize}
\item $\e{v}$ is a rearrangement of the shifted contents of $\bm{\beta}$,
\item  for each  $i$ and each pair  $z, z'$ of cells of  $\beta^{(i)}$ such that $z'$
lies immediately east or north of  $z$, the letter $\tilde{c}(z)$ occurs before  $\tilde{c}(z')$ in  $\e{v}$.
\end{itemize}
Define the following variant of the inversion statistic of \cite{HHLRU, Sami}:
\begin{align*}
\invi{k}(\e{v}) &= |\{(i,j)\mid \text{$i<j$ and $0<\e{v_i}-\e{v_j}<k$}\}|
\end{align*}
for any word $\e{v}$ in the alphabet of integers.
For example,
the 3-tuple $\bm{\beta}= \left(\partition{&~\\&},\partition{&~&~\\&~&},\partition{&&~\\&~&~}\right) =
(2,32,33)/(1,11,21)$
of skew shapes  has
shifted contents
\[{\tiny\left(\tableau{&3\\&},\tableau{&4&7\\&1&},\tableau{&&8\\&2&5}\right)}.\]
The word $\e{8341275}$ belongs to $\Wi{3}(\bm{\beta})$ and  $\invi{3}(\e{8341275}) = 5$.

%Let  $\SSYT_\lambda$ denote the set of semistandard Young tableaux of shape  $\lambda$.
\begin{corollary}\label{c main intro}
Let $\bm{\beta} = (\beta^{(0)},\beta^{(1)},\beta^{(2)})$ be a 3-tuple of skew shapes such that each  $\beta^{(i)}$ contains no  $2 \times 2$ square. The corresponding new variant $q$-Littlewood-Richardson coefficients are given by
\[\mathfrak{c}_{\bm{\beta}}^\lambda(q) =
\sum_{\substack{T \in \RSST_\lambda\\  \sqread(T) \in\Wi{3}(\bm{\beta})}}
q^{\invi{3}(\sqread(T))}.
\]
\end{corollary}
Note that if the set of shifted contents of  $\bm{\beta}$ is  $[n]$, since elements of  $\Wi{3}(\bm{\beta})$ contain no repeated letter,  $\RSST_\lambda$ can be replaced by
the set of standard Young tableaux of shape  $\lambda$ in Corollary~\ref{c main intro}.

This paper is part of the series \cite{BD0graph,BF}, which  uses ideas of \cite{LamRibbon,Sami,Sami2} to  generalize the Fomin-Greene theory \cite{FG} to quotients of  $\U$ with weaker relations than the quotients
considered in \cite{FG}.
%$\U/\Ilam{k}$ and other quotients of  $\U$.
The papers \cite{BD0graph, BF} are largely devoted to understanding the difficulties in pushing the approach in this paper beyond $k=3$.
%, and lay the foundations for part of the approach that can be extended.
In Section \ref{s generalizations}, we conjecture a strengthening of Theorem \ref{t J intro} to a quotient $\U/\Irkst{\leq 3}$ of $\U$  whose relations combine the Knuth relations and what we call \emph{rotation relations}:
\[ u_au_cu_b = u_bu_au_c \quad \text{ and } \quad u_cu_au_b = u_bu_cu_a \quad \text{ for $a<b<c$}.\]
Similar algebras are studied more thoroughly in \cite{BD0graph, BF}.
The idea of combining these two kinds of relations is due to Assaf \cite{Sami, Sami2} (see Remark \ref{r Sami KR}).
%??useful put this in ncdgraph
%the conjectured strengthening would yield a positive combinatorial formula for the Schur expansion of the generating functions of the components of the graphs $\G^{(k)}_{c,D}$ of \cite{Sami} in the  $k=3$ case (\cite{BD0graph} and \cite{BF} explain the relation between the graph theoretic and algebraic perspectives).
%(see e.g. \cite[\textsection5.3]{BD0graph}).
Though we have not been able to prove this strengthening, having it as a goal was crucial to our discovery and proof of Theorem \ref{t J intro}.
%By \cite[Proposition 14]{LamRibbon} and \cite{GH},  $\mathfrak{J}_\lambda(\mathbf{u})$ has a positive monomial expansion in  $\U/\Ilam{k}$ for all  $k$.
%Fomin and the author realized that if the Knuth relations $acb \equiv cab$ and $bac \equiv bca$, $a<b<c$, are replaced by the weaker relation  $(ac-ca)b \equiv b(ca-ac)$, then noncommutative versions of the elementary symmetric functions commute and thus some of the basic tools of \cite{FG} generalize to this setting.  This gives an approach to understanding the Schur expansions of the symmetric functions Assaf corresponding to components of  $D$-graphs studied by Assaf.

This paper is organized as follows: Section \ref{s Lam's algebra} introduces Lam's algebra of ribbon Schur operators  $\U/\Ilam{k}$ and defines LLT polynomials.  Section \ref{s reading} introduces the combinatorics  needed for the proof of the main theorem.  In Section \ref{s positive monomial}, we  state and prove a stronger, more technical version of the main theorem,  and we state and prove the full version of Corollary \ref{c main intro}.
In Section \ref{s generalizations}, we conjecture a strengthening of  the main theorem and  describe our progress towards proving it.

\section{Lam's algebra and LLT polynomials}
\label{s Lam's algebra}
We introduce Lam's algebra of ribbon Schur operators, reconcile a definition of LLT polynomials from \cite{Sami} with equivalence classes of words in this algebra, and reformulate this
definition in a way that is convenient for our main theorem.

\subsection{Diagrams and partitions}\label{ss diagram}

A \emph{diagram} or \emph{shape} is a finite subset of $\ZZ_{\ge 1}\times\ZZ_{\ge 1}$. A diagram is drawn as a set of  square cells in the plane with the English (matrix-style) convention so that row (resp. column) labels start with 1 and increase from north to south (resp. west to east).
The \emph{diagonal} or \emph{content} of a cell $(i,j)$ is $j-i$.

A partition $\lambda$ of $n$ is a weakly decreasing sequence $ (\lambda_1, \ldots, \lambda_l)$ of nonnegative integers that sum to $n$.
%, i.e. $n= \sum_{i=1}^l \lambda_i$.
The \emph{shape} of  $\lambda$ is the subset
$\{(r,c) \mid r \in [l], \ c \in [\lambda_r]\}$
of $\ZZ_{\geq 1} \times \ZZ_{\geq 1}$.
Write $\mu \subseteq \lambda$ if the shape of $\mu$ is contained in the shape of $\lambda$.
%?? subset already defined for diagrams by def
If $\mu \subseteq \lambda$, then $\lambda/\mu$ denotes the \emph{skew shape} obtained by removing the cells of $\mu$ from the shape of $\lambda$.
The \emph{conjugate partition} $\lambda'$ of $\lambda$ is the partition whose shape is the transpose of the shape of $\lambda$.

We will make use of the following partial orders on $\ZZ_{\ge 1}\times\ZZ_{\ge 1}$:
\begin{align*}
\text{$(r,c)\le_{\searrsub}(r',c')$  whenever $r\le r'$ and $c\le c'$,}\\
\text{$(r,c)\le_{\nearrsub}(r',c')$  whenever $r\ge r'$ and $c\le c'$.}
\end{align*}
 It will occasionally be useful to think of diagrams  as posets for the order  $<_\searrsub$ or $<_\nearrsub$.

\subsection{The algebra $\U$}

We will mostly work over the ring $\field=\QQA$ of Laurent polynomials in the indeterminate $\qlam$.
Let  $\U$ be the free associative  $\field$-algebra in the noncommuting variables $u_i$,  $i \in \ZZ$.
%Let $\U$ be the noncommutative associative algebra over $\field$ generated by $u_i$,  $i \in \ZZ$.
We think of the monomials of $\U$ as words in the alphabet of integers and frequently write  $\e{n}$ for the variable $u_n$.
We often abuse notation by writing  $\e{v}\in \U$ to mean that  $\e{v}$ is a monomial of $\U$.
%typically write $\e{v}=\e{v_1\cdots v_t}$ for a monomial.

\subsection{The nil-Temperley-Lieb algebra}
\label{ss The nil-Temperley-Lieb algebra}

The nil-Temperley-Lieb algebra \cite{BJS,FG} is the $\field$-algebra generated by $s_i$, $i\in\ZZ$, and relations
\begin{alignat*}{3}
&s_i^2 &&= ~~ 0, \\
&s_{i+1}s_is_{i+1} &&= ~~ 0, \\
&s_is_{i+1}s_i &&= ~~ 0, \\
&s_i s_j &&= ~~ s_j s_i \qquad \text{for $|i-j| > 1$.}
\end{alignat*}
The reduced words of 321-avoiding permutations of  $[n]$ form a basis for the subalgebra generated by $s_1,s_2,\dots,s_{n-1}$. Here and throughout the paper, $[n] := \{1,2,\dots,n\}$.

Let $\mathcal{P}$ denote the set of partitions and write $\field \mathcal{P}$ for the free  $\field$-module with basis indexed by $\mathcal{P}$.
By \cite{BJS} (see also \cite{FG}), there is a faithful action of the nil-Temperley-Lieb algebra on  $\field \mathcal{P}$, defined by
\[\nu \circ s_i = \begin{cases}
\mu \qquad & \text{if $\mu/\nu$ is a cell of content $i$,} \\
0 \qquad & \text{otherwise.}
\end{cases}\]
A \emph{skew shape with contents} is an equivalence class of skew shapes, where two skew shapes are equivalent if there is a content and
$<_\searrsub$-order preserving bijection between their diagrams.
By \cite[\textsection2]{BJS}, the map sending an element  $v = s_{i_1} \cdots s_{i_t}$ of the nil-Temperley-Lieb algebra to those skew shapes  $\mu/\nu$ such that  $\nu \circ v = \mu$ defines a bijection from the monomial basis of the nil-Temperley-Lieb algebra to skew shapes with contents.

\subsection{Lam's algebra of ribbon Schur operators}
\label{ss Lams algebra}

Lam defines \cite{LamRibbon} an algebra of ribbon Schur operators, which gives an elegant algebraic framework for LLT polynomials.  It  sets the stage for  applying the theory of noncommutative symmetric functions from \cite{FG}  to the problem of computing Schur expansions of LLT polynomials.
We have found it most convenient to work with the following variant\footnote{There is an algebra antiautomorphism from $\QQ(\qlam)\tsr_\field \U/\Ilam{k}$ to the algebra in \cite{LamRibbon},  defined  by sending
$u_i \mapsto u_i$ and setting $q = \qlam^{-2}$, where $\qlam$ denotes the $q$ from \cite{LamRibbon}.} of Lam's algebra.
Set  $q = \qlam^{-2}$.  Let  $\U/\Ilam{k}$ be the quotient of  $\U$ by the following relations (let $\Ilam{k}$ denote the corresponding two-sided ideal of $\U$):
\begin{alignat}{3}
&u_i^2 &&=~~ 0 \qquad &&\text{for $i\in\ZZ$,}   \\
&u_{i+k}u_iu_{i+k} &&=~~ 0 \qquad &&\text{for $i\in\ZZ$,}  \\
&u_iu_{i+k}u_i &&=~~ 0 \qquad &&\text{for $i\in\ZZ$,}   \\
&u_iu_j &&=~~ u_ju_i \qquad &&\text{for $|i-j| > k$,} \label{e far commute} \\
&u_iu_j &&=~~ q^{-1} u_ju_i \qquad &&\text{for $0<j-i<k$.} \label{e q commute}
\end{alignat}
We refer to \eqref{e far commute} as the \emph{far commutation} relations.

A $k$-ribbon is a connected skew shape of size $k$ containing no $2\times2$ square. The \emph{content} of  a ribbon is the maximum of the contents of its cells.
The \emph{spin} of a ribbon $R$, denoted $\spin(R)$, is the number of rows in the ribbon, minus 1.
By \cite{LamRibbon}, the following defines a faithful right action of  $\U/\Ilam{k}$ on $\field \mathcal{P}$:
\[\nu \cdot u_i = \begin{cases}
\qlam^{\spin(\mu/\nu)}\mu \qquad & \text{if $\mu/\nu$ is a $k$-ribbon of content $i$,} \\
0 \qquad & \text{otherwise.}
\end{cases}\]

For the variant of LLT polynomials we prefer to work with in this paper (see \textsection\ref{ss LLT}), it is better to work with an action on $k$-tuples of partitions.  We now describe this action and its relation to the one just defined.  Unfortunately, we only know how to relate these actions at  $\qlam =1$.

We briefly recall the definitions of  $k$-cores and  $k$-quotients (see \cite{HHLRU} for a more detailed discussion).
The  \emph{$k$-core} of a partition  $\mu$, denoted  $\core_k(\mu)$, is the unique partition obtained from  $\mu$ by removing  $k$-ribbons until it is no longer possible to do so.
Let  $\mu$ be a partition with $k$-core $\nu$.  Then for each  $i = 0,1,\dots,k-1$, there is exactly one way to add a  $k$-ribbon of content  $c_i \equiv i \bmod{k}$ to  $\nu$.
The \emph{$k$-quotient of $\mu$}, denoted $\quot_k(\mu)$, is  the  unique $k$-tuple $\bm{\gamma}=(\gamma^{(0)},\dots,\gamma^{(k-1)})$ of partitions such that the multiset of integers $k\cdot c(z)+c_i$ for $z\in\gamma^{(i)}$ and $i = 0,1,\dots,k-1$ is equal to the multiset of contents of the ribbons in any $k$-ribbon tiling of $\mu/\nu$.

%The \emph{$k$-quotient of $\mu$}, denoted $\quot_k(\mu)$, is  the  unique $k$-tuple $\bm{\gamma}=(\gamma^{(0)},\dots,\gamma^{(k-1)})$ of skew shapes such that (I) each $\gamma^{(i)}$ is a partition diagram translated so that its northwesternmost cell has content $(c_i-i)/k$ and lies in row 1 or column 1, and (II) the multiset of integers $k\cdot c(z)+i$ for $i = 0,1,\dots,k-1$ and $z\in\gamma^{(i)}$ is equal to the multiset of contents of the ribbons in any $k$-ribbon tiling of $\mu/\nu$.

%If $\mu/\nu$ can be tiled by $k$-ribbons, then $\core_k(\mu)=\core_k(\nu)$ and $\quot_k(\nu)\subset \quot_k(\mu)$, hence $\quot_k(\mu/\nu)=(\gamma^{(0)}/\delta^{(0)},\dots,\gamma^{(k-1)}/\delta^{(k-1)}$, where $\bm{\gamma} = \quot_k(\mu)$, $\bm{\delta} = \quot_k(\nu)$
%???do I need that quot_k is a bijection or surjection?

%Let $\phi_j^{(k)}$ be the homomorphism from $\U$ to the nil-Temperley-Lieb algebra given by
%\[u_i \mapsto \begin{cases}
%s_{(i-\rsd{i})/k} & \text{if $i \equiv j \bmod{k}$,} \\
%1 & \text{otherwise,}
%\end{cases}\]
%where $\rsd{i}$ denotes the element of $[k]$ congruent to $i\bmod{k}$.

%???might be cleaner to define this action for U/Ilam  and then say that it matches the other action at q=1
Let $\U/\Ilam{k}|_{\qlam=1}$ denote the algebra $\U/\Ilam{k}$ specialized to  $q=1$.
It is equal to the  $k$-fold tensor product over  $\QQ$ of the nil-Temperley-Lieb algebra.
For each  $\mathbf{d}=(d_0,\dots,d_{k-1}) \in \ZZ^n$, define
an action of $\U/\Ilam{k}|_{\qlam=1}$ on  $k$-tuples of partitions by
%??  $\field \mathcal{P}^{k}$
\begin{align*}
\bm{\delta} \circ_\mathbf{d} u_i = (\delta^{(0)}, \ldots, \delta^{(\rsd{i}-1)}, \delta^{(\rsd{i})} \circ s_{(i-\rsd{i})/k-d_{\rsd{i}}}, \delta^{(\rsd{i}+1)},\ldots),
\end{align*}
where $\rsd{i}$ denotes the element of $\{0,1,\dots,k-1\}$ congruent to $i\bmod{k}$.
This action is equivalent to the action of $\U/\Ilam{k}|_{\qlam=1}$ on partitions with  $k$-core  $\nu$ when  $d_i = (c_i-i)/k$ and the  $c_i$ are determined by  $\nu$ as above; the precise relation is  $\nu \cdot \e{v} = \qlam^a \mu$ for some  $a \in \ZZ$
if and only if  $\quot_k(\nu) \circ_\mathbf{d} \e{v}= \quot_k(\mu)$,  for any word $\e{v} \in \U$.

\subsection{Statistics on words and equivalence classes}
%Define
%[\SSYT(\bm{\beta}) = \SSYT(\beta^{(0)})\times\cdots\times\SSYT(\beta^{(k-1)})\]
%be the set of tuples $T = (T_1,\dots,T_k)$, where $T_i$ is a semistandard Young tableau of shape $\beta^{(i)}$. We call $T$ a semistandard Young tableau on the tuple of shapes $\bm{\beta}$.
%Define the \emph{content reading word} of a $k$-tuple of tableaux to be the word obtained by reading entries in increasing order of shifted content and reading diagonals in $\searr$ direction.

Following \cite{Sami}, we introduce statistics  $\Des$ and  $\inv$ on words and the set of  $k$-ribbon words, which will prepare us to define LLT polynomials.  We then relate these to the algebra $\U/\Ilam{k}$.

Given a word $w$ of positive integers (not thought of as an element of $\U$) and a weakly increasing sequence of integers $c$ of the same length, define the following statistics on pairs $(w,c)$
\begin{align*}
\Des_k(w,c) &= \{(i,j)\mid\text{$i<j$, $w_i>w_j$, and $c_j-c_i=k$}\}, \\
\inv_k(w,c) &= |\{(i,j)\mid\text{$i<j$, $w_i>w_j$, and $0<c_j-c_i<k$}\}|.
\end{align*}
These are called the \emph{$k$-descent set} and \emph{the $k$-inversion number of the pair $(w,c)$}.
Define the corresponding statistics on words  $\e{v}\in \U$:
\begin{align*}
\Desi{k}(\e{v}) &= \{(\e{v_i},\e{v_j})\mid\text{$i<j$ and $\e{v_i}-\e{v_j}=k$}\} \text{ (a multiset)}, \\
\invi{k}(\e{v}) &= |\{(i,j)\mid \text{$i<j$ and $0<\e{v_i}-\e{v_j}<k$}\}|.
\end{align*}

See Example~\ref{ex LLT}.
For a pair $(w,c)$ as above with  $w$  a permutation, define $\e{v}= (w,c)^{-1}$ by $\e{\e{v_i}}=c_j$ where $j$ is such that $w_j = i$.
Given $c$ as above and a  $k$-descent set  $D$, let  $(c,D)^{-1}$ denote the multiset  $\{(c_j,c_i) \mid (i,j) \in D\}$.
%These are called the \emph{$k$-descent set, the set of $k$-inversions, and the $k$-inversion number of the pair $(w,c)$}, respectively.
One checks that
\begin{align}\label{e inv des}
\Desi{k}((w,c)^{-1}) = (c,\Des_k(w,c))^{-1}\quad \text{and} \quad
 \invi{k}((w,c)^{-1}) = \inv_k(w,c).
\end{align}

\begin{definition}[Assaf \cite{Sami}]
A \emph{$k$-ribbon word} is a pair $(w,c)$ consisting of a word $w$ and a weakly increasing sequence of integers $c$
of the same length such that if $c_i=c_{i+1}$, then there exist integers $h$ and $j$ such that
$(h,i),(i+1,j)\in\Des_k(w,c)$ and $(i,j),(h,i+1)\not\in\Des_k(w,c)$.
In other words, $c_h=c_i-k$ and $w_i<w_h\le w_{i+1}$ while $c_j=c_i+k$ and $w_i\le w_j<w_{i+1}$.
\end{definition}

We now relate  $k$-ribbon words to the algebra $\U/\Ilam{k}$.

Say that two words  $\e{v},\e{v'}$ of  $\U$ are \emph{$k$-equivalent at  $q=1$} if  $\e{v}= \qlam^a \e{v'}$ in the algebra  $\U/\Ilam{k}$ for some  $a$.  Define
\emph{$q=1$ $k$-equivalence classes} in the obvious way; we omit the $k$ from these definitions when it is clear.  Note that  $\e{v}$ and  $\e{v'}$ are equivalent at  $q=1$ if and only if they are equal in the algebra  $\U/\Ilam{k}|_{\qlam=1}$: given that $\U/\Ilam{k}$  is defined by binomial and monomial relations, the only way this can fail is if $\e{v}\ne 0$ in  $\U/\Ilam{k}|_{\qlam=1}$ but  $f \e{v} = 0$ in  $\U/\Ilam{k}$ for some nonzero $f  \in \field$ (which must satisfy  $f|_{\qlam=1}=0$).  However, if this occurs then $\bm{\delta} \circ_\mathbf{0} \e{v} = \bm{\gamma}$ for some $k$-tuples  of partitions $\bm{\delta}, \bm{\gamma}$. Therefore, letting $\nu,\mu$ be shapes with  empty  $k$-cores such that  $\quot_k(\nu) = \bm{\delta}$,  $\quot_k(\mu)=\bm{\gamma}$, we have $\nu \cdot f\e{v} = f \qlam^a \mu$ for some  $a \in \ZZ$, hence  $f \e{v} \ne 0$ in  $\U/\Ilam{k}$.
%?? we're using that the action \circ  is faithful but not that \cdot  action is faithful.   this is important because there may be a slight omission in the proofs of Lam
It follows that  $\U/\Ilam{k}$ is a free  $\field$-module and the nonzero $q=1$  equivalence classes form an  $\field$-basis for  $\U/\Ilam{k}$.

\begin{propdef}\label{pd nonzero words}
The following are equivalent for a word $\e{v}\in\U$:
\begin{list}{\emph{(\roman{ctr})}}{\usecounter{ctr} \setlength{\itemsep}{2pt} \setlength{\topsep}{3pt}}
\item $\e{v}\ne 0$ in $\U/\Ilam{k}$,
\item $\e{v}\ne 0$ in $\U/\Ilam{k}|_{\qlam=1}$,
%\item $\phi_i^{(k)}(\e{v}) \ne 0$ for all $i\in[k]$,
\item for every pair $i<j$ such that $\e{v_i} = \e{v_j}$, there exists  $s,t$ such that $i<s<t<j$ and $\{\e{v_s},\e{v_t}\} = \{\e{v_i-k},\e{v_i+k}\}$,
\item $\e{v}= (w,c)^{-1}$ for some $k$-ribbon word $(w,c)$ with  $w$ a permutation.
\end{list}
If $\e{v}$ satisfies these properties, then we say it is a \emph{nonzero $k$-word}.
\end{propdef}
\begin{proof}
By the discussion above, (i) and (ii) are equivalent.
The equivalence of (ii) and (iii) follows from the bijection between the monomial basis of the nil-Temperley-Lieb algebra and skew shapes with contents (see \textsection\ref{ss The nil-Temperley-Lieb algebra}).  The equivalence of (iii) and (iv) is a straightforward unraveling of definitions.
\end{proof}

We record the following immediate consequence of the equivalence of (i) and (iii) for later use.
\begin{corollary}
\label{c temperley lieb}
Let $\e{v}= \e{v_1 \cdots v_t} \in \U$ be a word such that  $\e{v_1} = \e{v_t}$. If either $\e{v}$ does not contain $\e{v_1 + k}$ or $\e{v}$ does not contain $\e{v_1 - k}$, then $\e{v}= 0$ in  $\U/\Ilam{k}$.
\end{corollary}
%\begin{proof}
%We give the proof in the case $\e{v}$ does not contain $\e{v_1}+k$, the other case being similar.
%We have $\phi_{\e{v_1}}^{(k)}(\e{v}) = s_{i_1} s_{i_2} \cdots s_{i_t}$, where $i_1 = i_t$ and for $r = 2,3,\ldots,t-1$, $i_r$ is either $\le i_1$ or $\ge i_1 + 2$. By the commutation relations, $s_{i_1} s_{i_2} \cdots s_{i_t}$ is equal to the product of those $s_{i_j}$ such that $i_j \le i_1$ (in order) times the product of those $s_{i_j}$ such that $i_j \ge i_1 + 2$; this is 0 by Lemma \ref{l temperley lieb}.
%\end{proof}

A \emph{$k$-tuple of skew shapes with contents} is a $k$-tuple $\bm{\beta} =(\beta^{(0)},\dots,\beta^{(k-1)})$ such that each  $\beta^{(i)}$ is a skew shape with contents.
The \emph{shifted content} of a cell  $z$ of  $\bm{\beta}$ is
\[\tilde{c}(z) = k\cdot c(z)+i,\]
when $z \in \beta^{(i)}$ and where $c(z)$ is the usual content of $z$ regarded as a cell of $\beta^{(i)}$.
The \emph{content vector} of $\bm{\beta}$ is the weakly increasing sequence of integers consisting of the shifted contents of all the cells of $\bm{\beta}$ (with repetition).

\begin{definition}
\label{d Wi}
Define the following sets of words:
\begin{itemize}
\item $\WRib_k(c,D) = $ the set of $k$-ribbon words  $(w,c)$ with $k$-descent set $D$.
\item $\Wi{k}(c,D') = $ the set of nonzero $k$-words $\e{v}$ such that sorting $\e{v}$ in weakly increasing order yields $c$ and $\Desi{k}(\e{v}) = D'$.
\item $\Wi{k}(\bm{\beta})  = \{\e{v}\in \U \mid \bm{\delta} \circ_\mathbf{0}\e{v}=  \bm{\gamma}\}$ for any $k$-tuple $\bm{\beta} = \bm{\gamma}/\bm{\delta}$ of skew shapes with contents.
\end{itemize}
%Let $\mu/\nu$ be a skew shape that can be tiled by  $k$-ribbons such that $\bm{\beta} = \quot_k(\mu/\nu)$.
%Set $\Wi{k}(\bm{\beta})  = \{\e{v}\in \U \mid \nu \cdot\e{v}= \qlam^a \mu \text{ for some  $a$}\}$.
\end{definition}

%\begin{proposition}[\cite{Sami}]\label{p sami}
%Let $\bm{\beta}$ be a $k$-tuple of skew shapes with content vector $c$. Let $D=\Des_k(w,c)$, where $w$ is the shifted content reading word of any $\mathbf{T}\in\SSYT(\bm{\beta})$. Then $D$ depends only on $\bm{\beta}$ and the map from $\bm{\beta}$ to $c,D$ defines an injection from $k$-tuples of skew shapes to pairs consisting of a content vector and a $k$-descent set.
%\end{proposition}

For a word  $v$, the \emph{standardization} of  $v$, denoted $v^{\stand}$,
is the permutation obtained from $v$ by first relabeling, from left to right, the occurrences of the smallest letter in $v$ by  $1,\ldots,t$, then relabeling the occurrences of the next smallest letter of $v$ by $t+1,\ldots,t+t'$, etc.
\begin{example}\label{ex LLT}
Let
$\bm{\beta}$ be the 3-tuple $\left(\partition{&~\\&},\partition{&~&~\\&~&~},\partition{&&~\\&~&~}\right)$ of skew shapes with contents.  Its shifted contents are
\[{\tiny\left(\tableau{&3\\&},\tableau{&4&7\\&1&4},\tableau{&&8\\&2&5}\right)}.\]
%An element of $\SSYT(\bm{\beta})$ is
%\[\mathbf{T} = {\tiny\left(\tableau{&7\\&},\quad\tableau{&1&3\\&4&5},\quad\tableau{&&2\\&6&8}\right)}.\]
The content vector $c$ of $\bm{\beta}$, a word $\e{v}\in \Wi{3}(\bm{\beta})$, and $w=(\e{v}^\stand)^{-1}$ (note $\e{v} =(w,c)^{-1}$):
\begin{alignat*}{2}
&w &&= 46715832 \\
&c &&= 12344578 \\
&\e{v}&&= \e{48714235}.
\end{alignat*}
The statistics defined above, on the pair $(w,c)$ and on $\e{v}$:
\begin{alignat*}{2}
&\Des_3(w,c) &&= \{(1,4),(5,7),(6,8)\} \\
&\Desi{3}(\e{v}) &&= \{(4,1),(7,4),(8,5)\} \\
&\inv_3(w,c) &&= \invi{3}(\e{v})=6.
\end{alignat*}
We also have $(w,c)\in\WRib_3(c,D)$ for $D=\Des_3(w,c)$, and $\Wi{3}(\bm{\beta}) = \Wi{3}(c,D')$ for $D'=\Desi{3}(\e{v})$.
\end{example}

\begin{proposition}\label{p equivalence classes}
Let $\bm{\beta}$ be a  $k$-tuple of skew shapes with contents and let $c$ be its content vector.
%$c$ is a weakly increasing sequence of integers, and $D$ (resp.  $D'$) is a set (resp. multiset) of pairs of integers.
Let  $\tilde{c}^\stand$ be the function on the cells of $\bm{\beta}$ that assigns to the cells of shifted content  $i$ the letters of  $c^\stand$ that relabel the $i$'s in $c$, increasing in the  $\searr$ direction.
%Let  $\mathbf{T}$ be the tableau of shape  $\bm{\beta}$ filled with  the letters $c^\stand$ such that the diagonal of shifted content  $i$ is filled with the letters of  $c^\stand$ that relabel the $i$'s in $c$, increasing in the  $\searr$ direction.
\begin{list}{\emph{(\roman{ctr})}}{\usecounter{ctr} \setlength{\itemsep}{2pt} \setlength{\topsep}{3pt}}
\item The rule  $\bm{\beta} \mapsto \Wi{k}(\bm{\beta})$ defines a bijection between  $k$-tuples of skew shapes with contents and nonzero  $q=1$ equivalence classes.
\item Suppose  $D'$ is determined from $\bm{\beta}$ as follows:  the number of pairs of cells  $z,z'$ in  $\bm{\beta}$  such that $\tilde{c}(z)=\tilde{c}(z')+k$ and  $z <_\searrsub z'$ is the multiplicity of  $(\tilde{c}(z),\tilde{c}(z'))$ in  $D'$.
    Then  $\Wi{k}(\bm{\beta}) = \Wi{k}(c,D')$.
\item Suppose  $D$ is such that  $\WRib_k(c,D)$ is nonempty and let $D' = (c,D)^{-1}$.  Then we have the following bijection
\[\{(w,c) \in \WRib_k(c,D) \mid \text{$w$ a permutation}\}\to\Wi{k}(c,D'), ~(w,c) \mapsto (w,c)^{-1},\]
with inverse given by  $((\e{v}^\stand)^{-1},c)\mapsfrom\e{v}$.
\item If $\bm{\beta}$ and  $D'$ are related as in (ii), and $D' = (c,D)^{-1}$ as in (iii), then   $D$ is determined from $\bm{\beta}$ as follows:
$D$ is the set of pairs  $(\tilde{c}^\stand(z'),\tilde{c}^\stand(z))$ with $z,z'$ cells of $\bm{\beta}$ such that $\tilde{c}(z)=\tilde{c}(z')+k$ and  $z <_\searrsub z'$.
\end{list}
\end{proposition}
If $\bm{\beta}$ is as in Example \ref{ex LLT}, then  $c^\stand = 12345678$ and  $\tilde{c}^\stand$ is given by
\[{\tiny\left(\tableau{&3\\&},\tableau{&4&7\\&1&5},\tableau{&&8\\&2&6}\right)}.\]
\begin{proof}
By the discussion before Proposition-Definition \ref{pd nonzero words}, (i) reduces to the case  $k=1$. The case  $k=1$ follows from \cite{BJS}.
For (ii), the relations of  $\U/\Ilam{k}$ preserve the statistic $\Desi{k}(\e{v})$ on nonzero  $k$-words, hence  $\Wi{k}(c,D')$ is a union of $q=1$ equivalence classes.  That  $\Wi{k}(c,D')$ is the single equivalence class  $\Wi{k}(\bm{\beta})$ follows from the observation that a skew shape with contents is determined by its multiset of contents and, for each of its pairs of diagonals $G, G'$ with the content of  $G$ 1 more than that of  $G'$, the number of pairs of cells  $(z,z')$,  $z \in G$, $z' \in G'$ such that  $z <_\searrsub z'$.
Statements (iii) and (iv) are straightforward from definitions, Proposition-Definition \ref{pd nonzero words}, and \eqref{e inv des}.
\end{proof}

\subsection{LLT polynomials}
\label{ss LLT}
LLT polynomials are certain  $q$-analogs of products of skew Schur functions, first defined by Lascoux, Leclerc, and Thibon in \cite{LLT}.
There are two versions of LLT polynomials (which we distinguish following the notation of \cite{GH}): the combinatorial LLT polynomials of \cite{LLT} defined using spin, and the new variant combinatorial LLT polynomials of \cite{HHLRU} defined using inversion numbers (we called these LLT polynomials in the introduction). % to distinguish the two types of LLT polynomials.
Although Lam's algebra is well suited to studying the former, we prefer to work with the latter because they can be expressed entirely in terms of words, and because inversion numbers are easier to calculate than spin.
%The LLT polynomials and  $q$-Littlewood-Richardson coefficients in the introduction
%For simplicity in the introduction, we referred to the

Let  $\e{v} = \e{v_1 \cdots v_t}$ be a word.
We write $\Des(\e{v}) := \{i \in [t-1] \mid \e{v_i} > \e{v_{i+1}}\}$ for the \emph{descent set} of $\e{v}$.
Let
\[Q_{\Des(\e{v})}(\mathbf{x}) = \sum_{\substack{1 \le i_1 \le \, \cdots \, \le i_t\\j \in \Des(\e{v}) \implies i_j < i_{j+1} }} x_{i_1}\cdots x_{i_t}\]
be Gessel's \emph{fundamental quasisymmetric function} \cite{GesselPPartition} in the commuting variables $x_1,x_2,\ldots$.
%considered as an element of the polynomial ring  $\QQA[x_1,x_2,\ldots]$.

The result \cite[Corollary 4.3]{Sami}, which we take here as a definition, expresses the new variant combinatorial LLT polynomials of \cite{HHLRU} in terms of  $k$-ribbon words.
%Our starting point is the result \cite[Corollary 4.3]{Sami}, which expresses the new variant combinatorial LLT polynomials of \cite{HHLRU} in terms of  $k$-ribbon words.
\begin{definition}\label{d new variant LLT}
Let $\bm{\beta}$ be a  $k$-tuple of skew shapes with contents, and let $c,D$ be the corresponding content vector
and $k$-descent set from Proposition~\ref{p equivalence classes} (iv).
The \emph{new variant combinatorial LLT polynomials} are the generating functions
\[\mathcal{G}_{\bm{\beta}}(\mathbf{x};q) = \sum_{\substack{(w,c)\in\WRib_k(c,D)\\ \text{$w$ a permutation}}}q^{\inv_k(w,c)}Q_{\Des(w^{-1})}(\mathbf{x}).\]
\end{definition}

%We now reformulate Corollary~\ref{c assaf} in terms of words inverse to $w$ in order to reconcile the setup of Assaf with that of Lam and Fomin-Greene.
For this paper, we have found the following expressions for LLT polynomials to be the most useful.  The second involves only words and the statistics  $\invi{k}$ and  $\Desi{k}$.  Also, as will be seen in \textsection\ref{ss qLR coefs}, these expressions allow for the application of the machinery of \cite{FG}.
\begin{proposition}\label{p assafi}
Let  $\bm{\beta}$ be a  $k$-tuple of skew shapes with contents and let $c,D'$ be determined from  $\bm{\beta}$ as in Proposition \ref{p equivalence classes} (ii).
Then
\[\mathcal{G}_{\bm{\beta}}(\mathbf{x};q) = \sum_{\e{v}\in\Wi{k}(\bm{\beta})}q^{\invi{k}(\e{v})}Q_{\Des(\e{v})}(\mathbf{x})  =  \sum_{\e{v} \in\Wi{k}(c,D')}q^{\invi{k}(\e{v})}Q_{\Des(\e{v})}(\mathbf{x}).\]
%\[\mathcal{G}_{\bm{\beta}}(\mathbf{x};q) = \langle \nu \cdot \sum_{\e{v}\in \U} Q_{\Des(\e{v})}(\mathbf{x})\e{v}, \mu \rangle  = \sum_{\e{v} \in\Wi{k}(c,D')}q^{\invi{k}(\e{v})}Q_{\Des(\e{v})}(\mathbf{x}).\]
\end{proposition}
\begin{proof}
This follows from Definition \ref{d new variant LLT}, Proposition \ref{p equivalence classes}, \eqref{e inv des}, and the fact
$\Des(\e{v}^\stand) = \Des(\e{v})$.
%Let $c$ be the content vector of  $\bm{\beta}$ and let  $D$ and  $D'$ be .  As  $w$ ranges over permutations such that $(w,c) \in \WRib_k(c,D)$,  $(w,c)^{-1}$ ranges over the elements of  $\Wi{k}(\bm{\beta}) = \Wi{k}(c,D')$
\end{proof}

\begin{remark}\label{r spin LLT}
The \emph{combinatorial LLT polynomials} $G^{(k)}_{\mu/\nu}(\mathbf{x};\qlam)$ are the  $\qlam$-generating functions over  $k$-ribbon tableaux weighted by spin
(\cite{LT00, LamRibbon, GH} use spin, whereas \cite{LLT} uses cospin).
%??\cite{HHLRU} uses sp, which is neither
We refer to \cite{GH} for their precise definition.  Lam shows \cite{LamRibbon} that
\[G^{(k)}_{\mu/\nu}(\mathbf{x};\qlam) = \langle\nu \cdot \Omega(\mathbf{x}, \mathbf{u}),\mu\rangle,\]
where $\langle\cdot ,\cdot  \rangle$ is the symmetric bilinear form on $\field \mathcal{P}$ for which  $\mathcal{P}$ is an orthonormal basis,
and
\[ \Omega(\mathbf{x}, \mathbf{u}) = \prod_{j=1}^\infty \prod_{i=-\infty}^{\infty}(1-x_ju_i)^{-1}
= \sum_{\e{v}\in \U} Q_{\Des(\e{v})}(\mathbf{x})\e{v}
%=\sum_{\lambda} s_\lambda(\mathbf x)\mathfrak J_{\lambda}(\mathbf{u})
\]
is a noncommutative Cauchy product in which the  $x_j$ commute with the  $u_i$.

The two types of LLT polynomials are related as follows.
Suppose $\mu/\nu$ can be tiled by $k$-ribbons. Let
$\bm{\beta}$ be the  $k$-tuple of skew shapes with contents obtained from $\quot_k(\mu)/\quot_k(\nu) = (\gamma^{(0)}/\delta^{(0)},\dots, \gamma^{(k-1)}/\delta^{(k-1)})$ by translating each $\gamma^{(i)}/\delta^{(i)}$ east by $(c_i-i)/k$, where the $c_i$ are determined by the  $k$-core of  $\mu$ ($=\core_k(\nu)$) as in \textsection\ref{ss Lams algebra}.  Then  by \cite{HHLRU} (see e.g. \cite[Proposition 6.17]{GH}),
 there is an integer $e$ such that (recall $q=\qlam^{-2}$)
\[\mathcal{G}_{\bm{\beta}}(\mathbf{x};q) = \qlam^eG_{\mu/\nu}^{(k)}(\mathbf{x};\qlam).\]
\end{remark}

\section{Reading words for $\mathfrak{J}_\lambda(\mathbf{u})$ when $k=3$}
\label{s reading}
Here we introduce new kinds of tableaux and reading words that arose naturally in our efforts to write $\mathfrak{J}_\lambda(\mathbf{u})$ as a positive sum of monomials in $\U/\Ilam{3}$.
The main new feature of these objects is that they involve posets obtained from posets of diagrams by adding a small number of covering relations.
We are hopeful that this will become part of a more general theory that extends tools from tableaux combinatorics to posets more general than partition diagrams.

In Sections \ref{s reading} and \ref{s positive monomial}, we write $f\equiv g$ to mean that $f$ and $g$ are equal in $\U/\Ilam{k}$, when the value of $k$ is clear from context (typically it is 3 or arbitrary).
\subsection{Tableaux} \label{ss tableaux}
Let  $\theta$ be a diagram (see \textsection\ref{ss diagram}).
A \emph{tableau of shape $\theta$} is the diagram $\theta$ together with an integer in each of its cells.
The \emph{size} of a tableau $T$, denoted  $|T|$, is the number of cells of $T$, and $\sh(T)$ denotes the shape of $T$.
For a tableau $T$ and a set of cells $S$ such that $S\subseteq\sh(T)$,
$T_S$ denotes the subtableau of $T$ obtained by restricting $T$ to the diagram $S$.
If $z$ is a cell of $T$, then $\e{T_z}$ denotes the entry of $T$ in $z$.
%??needed
When it is clear, we will  occasionally identify a tableau entry with the cell containing it.

If $T$ is a tableau,  $\e{m}$ a letter, and  $z = (r,c)$ is a cell not belonging to $\sh(T)$, then $T \sqcup {\tiny \tableau{\e{m}}}_{\, r,c}$ denotes the result of adding the cell $z$ to $T$  and filling it with $\e{m}$.
%denote the result of adding a new cell containing $\alpha$ to $T$ in the $r$-th row and $c$-th column.

A \emph{standard Young tableau} (SYT) is a tableau  $T$ of partition shape filled with the entries $1,2,\dots,|T|$ such that entries increase from north to south in each column and from west to east in each row.  The set of standard Young tableaux of shape $\lambda$ is denoted $\SYT_\lambda$.
%"Illustrations in such terms as \emph{row}..." ?? I decided I no longer need to be careful about South North versus down up

\subsection{Restricted shapes and restricted square strict tableaux}

\begin{definition}
A \emph{restricted shape} is a lower order ideal of a partition diagram for the order $<_\nearrsub$.
We will  typically specify a restricted shape as follows:
for any weak composition  $\alpha= (\alpha_1,\ldots, \alpha_l)$, let  $\alpha'$ denote the diagram  $\{(r,c) \mid c \in [l], \ r \in [\alpha_c] \}$.
Now let $\lambda = (\lambda_1,\dots,\lambda_l)$ be a partition and $\alpha = (\alpha_1,\dots,\alpha_l)$ a weak  composition such that $0 \le \alpha_1 \le \cdots \le \alpha_{j'}$, $\alpha_1 < \lambda_1,\ \alpha_2 < \lambda_2,  \ldots, \alpha_{j'} < \lambda_{j'}$, and
 $\alpha_{j'+1} = \lambda_{j'+1}, \ldots, \alpha_{l} = \lambda_{l}$  for some  $j' \in \{0,1,\dots,l\}$.
Then the set difference of $\lambda'$ by $\alpha'$, denoted $\lambda' \setminus \alpha'$, is a restricted shape and any restricted shape can be written in this way.
\end{definition}

Note that, just as for skew shapes, different pairs $\lambda, \alpha$ may define the same restricted shape $\lambda' \setminus \alpha'$.
An example of a restricted shape is
\[ {\footnotesize (65544444221)' \setminus (01222223221)'} \ =\  \partition{~\\~&~\\~&~&~&~&~&~&~\\~&~&~&~&~&~&~&~\\~&~&~\\~}. \qquad\]

\begin{definition}\label{d RSST}
A \emph{restricted tableau} is a tableau whose shape is a restricted shape.
A \emph{restricted square strict tableau} (RSST) is a restricted tableau such that entries
\begin{itemize}
\item strictly increase from north to south in each column,
\item strictly increase from west to east in each row,
\item satisfy $R_z +3\le R_{z'}$ whenever $z <_\searrsub z'$ and  $z, z'$ do not lie in the same row or column.
\end{itemize}
%of shape $\lambda' \setminus \alpha'$ is a restricted tableau of the form $T_{\lambda' \setminus \alpha'}$ for some SYT $T$ of shape $\lambda'$.
\end{definition}

For example,
\begin{align*}
\tiny
\tableau
{1\\
3&4\\
5&6&7&8&9&10\\
7&8&9&10&11&12&13\\
8&10&12
\\11}
\end{align*}
is an RSST of shape $(6554444)' \setminus (0122223)'$.

\subsection{Square respecting reading words}
\label{ss Square respecting reading words}
%\begin{proposition}
%Let $\lambda \vdash n$, let $R$ be a restricted tableau of shape $\lambda' \setminus \alpha'$, and let $X$ be the subset of $[n]$ consisting of the letters not in $R$.
%Let $x_1<x_2<\dots<x_{|X|}$ denote the elements of $X$.
%Let $r_1,\dots,r_l$ be the entries in the northern border of $R$.
%Then $R$ is an RST if and only if $r_i > x_{a_i}$ for all $i\in[l]$, where $a_i=\sum_{k=1}^i \alpha_k$.
%\end{proposition}

%A \emph{standard $\nearr$-tableau} of shape $\theta$ is a linear extension of the poset $\theta$ for the $<_\nearrsub$-order, depicted as a tableau of shape $\theta$ with entries that increase from south to north in each column and increase from west to east in each row.
%By definition, the reading words of a tableau $R$ with no repeated entry are in bijection with the standard $\nearr$-tableaux of shape $\sh(R)$:
%if  $Q$ is a standard $\nearr$-tableaux of shape $\sh(R)$,
%then the corresponding reading word is obtained by reading the entries of  $R$ in the order given by  $Q$ (so that the leftmost letter of the word corresponds to the 1 in $Q$).
%%then the letter in position  $i$ of the corresponding reading word is $\e{R_z}$, where $z$ is the cell such that $Q_z=i$.
%For example,
%the  $\nearr$-tableau corresponding to the reading word \eqref{e reading word} is
%\[
%\tiny\tableau{
%14\\
%12&13\\
%10&11&15&16&17&21&22\\
%3&5&7&8&9&18&19&20\\
%2&4&6\\
%1}
%\]

\begin{definition}\label{d arrows}
An \emph{arrow square}  $S$ of an RSST $R$ is a subtableau of $R$ such that $\sh(S)$ is the intersection of  $\sh(R)$ with a  $2\times2$ square,
%(of shape $\partition{~&~\\~&~}$ or $\partition{~\\~&~}$
and  $S$ is of the form
\setlength{\cellsize}{3.4ex}
\[\tiny
\tableau{{\color{black}\put(.8,-1.6){\vector(-1,1){.6}}}\put(-.1,-0.96){$a$}&b\\a\text{+}2&a\text{+}3}\qquad
\tableau{{\color{black}\put(.8,-1.6){\vector(-1,1){.6}}}\put(-.1,-0.96){$a$}\\a\text{+}2&a\text{+}3}\qquad\qquad
\tableau{{\color{black}\put(.25,-1.1){\vector(1,-1){.6}}}\put(-.1,-1){$a$}&b\\a\text{+}1&a\text{+}3}\qquad
\tableau{{\color{black}\put(.25,-1.1){\vector(1,-1){.6}}}\put(-.1,-1){$a$}\\a\text{+}1&a\text{+}3},\]
with $b \in \{a+1,a+2\}$.
The first two of these are called \emph{$\nwarr$ arrow squares} and the last two are \emph{$\searr$ arrow squares}. Also, \emph{an arrow} of  $R$ is a directed edge between the two cells of an arrow square,  as indicated in the picture; we think of the arrows of  $R$ as the edges of a directed graph with vertex set the cells of $R$.
\setlength{\cellsize}{2.1ex}
\end{definition}

\begin{definition}\label{d square reading word}
A \emph{reading word} of a tableau $R$ is a word $\e{w}$ consisting of the entries of $R$ such that for any two cells $z$ and $z'$ of $R$ such that $z <_{\nearrsub} z'$, $\e{R_z}$ appears to the left of $\e{R_{z'}}$ in $\e{w}$.

A \emph{square respecting reading word} $\e{w}$ of an RSST $R$ is a reading word of $R$ such that for each arrow of $R$, the tail of the arrow appears to the left of the head of the arrow in $\e{w}$.

There is no reason to prefer one square respecting reading word over another (see Theorem \ref{t square respecting connected}), but it is useful to have notation for one such word.
So we define, for any RSST  $R$, the word $\sqread(R)$ as follows: let $D^1,D^2,\dots,D^t$ be the diagonals of $R$, starting from the southwest. Let $\e{w^i}$ be the result of reading, in the $\nwarr$ direction, the entries of $D^i$ that are $\nwarr$ arrow tails followed by the remaining entries of $D^i$, read in $\searr$ direction. Set $\sqread(R)=\e{w^1}\e{w^2}\cdots \e{w^t}$.  It is easy to check that $\sqread(R)$ is a square respecting reading word of $R$.
\end{definition}

\begin{example}\label{ex square respecting}
Here is an RSST drawn with its arrows:
\setlength{\cellsize}{3.4ex}
\[R = \tiny\tableau{{\color{black}\put(.8,-1.6){\vector(-1,1){0.6}}}\put(-.1,-0.94){$1$}\\
{\color{black}\put(.8,-1.6){\vector(-1,1){0.6}}}\put(-.1,-0.94){$3$}&4\\
5&6&{\color{black}\put(.8,-1.6){\vector(-1,1){0.6}}}\put(-.24,-0.94){$16$}&17&25&{\color{black}\put(.25,-1.1){\vector(1,-1){.6}}}\put(-.24,-1){$31$}&{\color{black}\put(.25,-1.1){\vector(1,-1){.6}}}\put(-.24,-1){$33$}\\
{\color{black}\put(.8,-1.6){\vector(-1,1){0.6}}}\put(-.24,-0.94){$10$}&11&18&19&26&32&34&36\\
12&13&20\\
15}\]
\setlength{\cellsize}{2.1ex}
Of the following three reading words of $R$, the first two are square respecting, but the last is not.
\begin{align*}
& \e{15~12~13~10~5~11~20~6~3~18~19~4~1~16~17~26~25~32~31~34~33~36} = \sqread(R)\\
& \e{15~12~13~10~5~11~20~6~3~4~1~18~19~16~17~26~25~32~31~34~33~36}\quad\text{\footnotesize (square respecting)} \\
& \e{15~12~10~13~11~20~18~19~26~5~6~3~4~1~16~17~25~32~34~36~31~33}\quad\text{\footnotesize (not square respecting)}
\end{align*}
\end{example}

A \emph{$\nearr$-maximal cell} of a diagram is a cell that is maximal for the order  $<_\nearrsub$.
A \emph{nontail removable cell} of an RSST $R$ is a $\nearr$-maximal cell of $R$ that is not the tail of an arrow of $R$.
The last letter of a square respecting reading word of $R$ must lie in a nontail removable cell of $R$.

\begin{remark}
\label{r readind words}
If a tableau has no repeated letter, then its entries are naturally in bijection with the letters in any of its reading words.
We can force this to be true for tableaux with repeated entries
by adding subscripts to repeated entries and matching subscripts to the letters of its reading words.
We will abuse notation and do this without saying so explicitly.
So for example, if  $\e{w}$ is a reading word of  $R$ and  $z$ a cell of  $R$,
then by \emph{the letter $\e{R_z}$ of  $\e{w}$} we really mean the unique subscripted letter of $\e{w}$ that
is equal to the subscripted version of $\e{R_z}$.
Also, if  $\e{vw}$ is a  reading word of $R$, then \emph{the canonical subtableau} of  $R$ with
reading word  $\e{v}$ is  the subtableau of $R$  whose subscripted entries are the subscripted versions of the letters of $\e{v}$.
\end{remark}

%\begin{remark}
%It is a consequence of Theorem~\ref{t gen square respecting connected} that for an SYT $T$, the reading words of $T$ are a subset of the set of words that have insertion tableau $T$.
%This fact is implicit in \cite{FG93}.
%%?? See Corollary 4.2 \cite{FG93}
%Moreover, every word with insertion tableau $T$ is a reading word if and only if $\sh(T)$ is a rectangle.
%\end{remark}

\begin{remark}\label{r L diagram}
All of the results of
%??pretty sure about this
\textsection\ref{ss Square respecting reading words} and \textsection\ref{ss Combinatorics of restricted square strict tableaux} about  restricted shapes and RSST can be generalized to the following setting:
a \emph{$\corner$-diagram} is a diagram  $\theta$ such that for any two cells $z = (r,c),$ $z'=(r',c')$ of  $\theta$ such that  $z <_\searrsub z'$, the cell $(r',c)$  is also in $\theta$.
A  \emph{$\corner$-tableau} is a tableau whose shape is a $\corner$-diagram, and a \emph{square strict $\corner$-tableau} is a  $\corner$-tableau satisfying the three conditions from
Definition \ref{d RSST}.
An \emph{arrow square} of a $\corner$-tableau $R$ is a subtableau of $R$ whose shape is an interval $[z,z']$ for the poset $\sh(R)$ with the order $<_\searrsub$ such that $z$ and $z'$ do not lie in the same row or column and $R_{z'}-R_z=3$.
%Generalizing the definition of arrows in Definition \ref{d arrows} to this setting requires a slight modification, in which an arrow square is no longer
%required to lie in a 2$\times$2 square, however it still consists of 3 or 4 cells and the arrow goes between two entries that differ by 3.

We have little use for this generality in this paper, so we do not discuss it further except to note one useful application.
Unlike restricted diagrams,  the set of   $\corner$-diagrams is closed under the operation of reflecting across a line in the direction  $\nearr$. Also, the set of square strict   $\corner$-tableaux with entries in  $[n]$ is closed under the operation of reflecting across a line in the direction  $\nearr$ and sending the entry  $a$ to  $n+1-a$.
Arrows of a square strict  $\corner$-tableau are reflected along with the cells.
\end{remark}

\subsection{Combinatorics of restricted square strict tableaux}
\label{ss Combinatorics of restricted square strict tableaux}

We assemble some basic results about RSST, square respecting reading words, and the images of these words in $\U/\Ilam{3}$. These are needed for the proof of the main theorem.
\begin{proposition}\label{p forbidden arrows}
The following configuration of arrows cannot occur in an RSST:
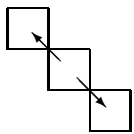
\setlength{\cellsize}{3ex}
{\tiny
\begin{alignat*}{6}
\tableau{{\color{black}\put(.8,-1.6){\vector(-1,1){.7}}}\\&{\color{black}\put(.2,-1.1){\vector(1,-1){.7}}}\put(-.1,-0.96){}\\&&~}
\end{alignat*}
}
\setlength{\cellsize}{2.1ex}
\end{proposition}
\setlength{\cellsize}{3ex}
\begin{proof}
Let  $z_1,z_2$ be the two cells of  $R$ as shown.
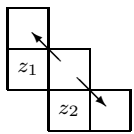
\begin{alignat*}{6}
{\tiny\tableau{{\color{black}\put(.8,-1.6){\vector(-1,1){.7}}}\\z_1&{\color{black}\put(.2,-1.1){\vector(1,-1){.7}}}\put(-.1,-0.96){}\\&z_2&~}}
\end{alignat*}
If this configuration occurs, then $R_{z_1} = R_{z_2}-2$, contradicting the  definition of an RSST.
\end{proof}
\setlength{\cellsize}{2.1ex}

\begin{lemma}
\label{l exist no tail corner}
Let  $R$ be an RSST  and $z$ a cell of  $R$ such that the only cells $>_\nearrsub z$ lie in the same column as $z$, and  $z$ is not the tail of a $\searr$ arrow.
Then  $R$  has a nontail removable cell that is weakly north of $z$.

Similarly, if $z$ is a cell of $R$ such that the only cells $>_\nearrsub z$ lie in the same row as $z$, and  $z$ is not the tail of a $\nwarr$ arrow,
then  $R$  has a nontail removable cell that is weakly east of $z$.
%Hence an RSST has at least one nontail removable cell and at least one square respecting reading word.
\end{lemma}
As an example of the second statement, if  $R$ is as in Example~\ref{ex square respecting} and  $z$ is the cell containing  $31$, then  there is exactly one nontail removable cell weakly east of  $z$, the cell containing  $36$.

\begin{proof}
We prove only the first statement.  The second then follows from Remark \ref{r L diagram}.
Let  $D$ be the  northeasternmost diagonal of   $R$ that has  nonempty intersection with the cells of  $R$ weakly north of  $z$.  Let $G$
be the directed graph with vertex set $D$ and edges given by the arrows of $R$.  Let  $D^N$ be the cells of  $D$ that are weakly  north of $z$.  Then  the underlying undirected graph of $G$  is contained in a path and there is no edge from  $D^N$  to vertices not in $D^N$.   Hence $D^N$  contains a vertex that is not an arrow tail and hence is a nontail removable cell of  $R$.
\end{proof}

\begin{lemma}\label{l small letter end}
Let  $R$ be an RSST  and $z$ a cell of  $R$ such that the only cells $>_\nearrsub z$ lie in the same column as $z$, and  $z$ is not the tail of a $\searr$ arrow.
Then there is a square respecting reading word  $\e{w}$ of  $R$ such that the letters appearing to the right of $\e{R_z}$ in  $\e{w}$ are $< R_z$.

Similarly, if  $z$ is a cell of  $R$ such that the only cells $>_\nearrsub z$ lie in the same row as $z$, and  $z$ is not the tail of a $\nwarr$ arrow, then there is a square respecting reading word  $\e{v}$ of  $R$ such that the letters appearing to the right of $\e{R_z}$ in  $\e{v}$ are $> R_z$.
%??? slight abuse of notation R_z here for repeated letters.
\end{lemma}
\begin{proof}
This follows by induction on  $|R|$ using Lemma \ref{l exist no tail corner}.
%We prove only the first statement.  The second then follows from Remark \ref{r L diagram}.
%The proof is by induction on  $|R|$.  By Lemma \ref{l exist no tail corner}, there is a nontail removable cell  $z'$ that is weakly north of  $z$.  The letter  $\e{R_{z'}}$ is the end  of a square respecting reading word of  $R$.  If  $z' = z$, then we are done.  Otherwise, by the first assumption on  $z$,  $R_{z'} < R_z$.  This, together with the lemma  applied inductively to $R-z'$,  yields the desired result.
\end{proof}

The next result gives a natural way to associate an element of $\U/\Ilam{3}$ to any RSST.
%???is this injective? surjective?
%The next result gives a natural way to associate to any RSST $R$ an element of $\U/\Ilam{3}$, namely, associate $R$ to any of its square respecting reading words.
\begin{theorem}\label{t square respecting connected}
Any two square respecting reading words of an RSST $R$ are equal in $\U/\Ilam{3}$.
\end{theorem}
\begin{proof}
Let $\R_R$ denote the graph with vertex set the square respecting reading words of $R$ and an edge for each far commutation relation.
We prove that  $\R_R$ is connected by induction on $|R|$.
Let $z_1, \ldots, z_t$ denote the nontail removable cells of  $R$
($t \geq 1$ by Lemma \ref{l exist no tail corner}).
By induction, each $\R_{R-z_i}$ is connected.  Hence  the induced subgraph of $\R_R$  with  vertex set consisting of those words that end in $\e{R_{z_i}}$, call it $\R_{R-z_i} \e{R_{z_i}}$, is connected. Let $H$ be the graph (with $t$ vertices) obtained from  $\R_R$ by contracting the subgraphs  $\R_{R-z_i}\e{R_{z_i}}$.
We must show that  $H$ is connected.

It follows from the definition of an RSST that if  $z,z'$  are $\nearr$-maximal cells of $R$, then  $|R_z - R_{z'}| \le 3$  if and only if   there is an arrow between $z$ and  $z'$.  Hence the $\e{R_{z_i}}$ pairwise commute.   Since  $\R_R$ contains a word ending in  $\e{R_{z_i}}\e{R_{z_j}}$ for every  $i\ne j$, $H$  is a complete graph.
\end{proof}

Lemma~\ref{l small letter end} and Theorem~\ref{t square respecting connected} have the following useful consequence.

\begin{corollary} \label{c removable 0}
If  $R$ is an RSST and  $z$  is a $\nearr$-maximal cell of  $R$, then  $\e{v}\e{R_z} =0$ in  $\U/\Ilam{3}$ for every square respecting reading word  $\e{v}$ of $R$.
\end{corollary}
\begin{proof}
By Theorem \ref{t square respecting connected},  it suffices to exhibit a single  square respecting reading word $\e{v}$ of  $R$ such that $\e{v}\e{R_z} \equiv0$.
By Proposition \ref{p forbidden arrows} one of the two statements of Lemma \ref{l small letter end} applies, thus $R$ has a square respecting reading word ending in  $\e{R_z}$ followed by letters all $> R_z$ or all  $< R_z$.  Hence Corollary \ref{c temperley lieb} with repeated letter  $\e{R_z}$ yields
 $\e{v}\e{R_z} \equiv 0$.
\end{proof}

Say that  an RSST  $R$ is \emph{nonzero} if any (equivalently, every) square respecting reading word of  $R$ is nonzero in  $\U/\Ilam{3}$ (see Proposition-Definition \ref{pd nonzero words}).
Although we will not need the next proposition in the proof of the main theorem, it is useful because it constrains the form of nonzero RSST.
\begin{proposition}\label{p squares are strict}
If an RSST $R$ contains an arrow square of the form
\setlength{\cellsize}{3.4ex}
\begin{equation}
\label{ep squares are strict}
{\tiny\tableau{{\color{black}\put(.25,-1.1){\vector(1,-1){.6}}}\put(-.1,-1){$a$}&a\text{\rm+}1\\a\text{\rm+}1&a\text{\rm+}3}} \qquad\text{or}\qquad
{\tiny\tableau{{\color{black}\put(.8,-1.6){\vector(-1,1){.6}}}\put(-.1,-0.96){$a$}&a\text{\rm+}2\\a\text{\rm+}2&a\text{\rm+}3}},
\end{equation}
\setlength{\cellsize}{2.1ex}
then every square respecting reading word of $R$ is 0 in $\U/\Ilam{3}$.
\end{proposition}
\begin{proof}
By Remark \ref{r L diagram}, we may assume without loss of generality that  $R$ contains an arrow square $S$ of the form on the left of \eqref{ep squares are strict}, with cells labeled as follows
%\setlength{\cellsize}{3.4ex}
%\[\tiny\tableau{{\color{black}\put(.25,-1.1){\vector(1,-1){.6}}}\put(-0.25,-1){$s_1$}&s_3\\s_2&s_4}.\]
%\setlength{\cellsize}{2.1ex}
\[S = {\tiny\tableau{s_1&s_3\\s_2&s_4}} \subseteq R.
\]
It suffices to exhibit one square respecting reading word of  $R$ that is 0 in $\U/\Ilam{3}$.

By induction on  $|R|$, we may assume  $s_3$ is the only nontail removable cell of  $R$ and  $S$ is the only arrow square of  $R$ of either of the forms in \eqref{ep squares are strict}.
Since  $s_3$ cannot be the head of a  $\searr$ arrow of  $R$, it follows from Lemma \ref{l exist no tail corner} that if  $R$
has a cell north of $s_1$ and in the same column as  $s_1$,
%??wordy for L-diagram generalize
then there is a nontail removable cell to the north of  $s_1$; we are assuming there is no such cell, so any cell  $>_\nearrsub s_1$ can only be in the same row as  $s_1$.
By Proposition \ref{p forbidden arrows},  $s_1$ is not the tail of a  $\nwarr$ arrow.
Hence by Lemma~\ref{l small letter end} (with  $z = s_1$), there is a square respecting reading word $\e{vR_{s_1} w}$ of $R$ such that
the letters of $\e{w}$ are $> R_{s_1} = a$; also we must have $\e{R_{s_3}}=\e{a+1}$ and $\e{R_{s_4}}$ contained in $\e{w}$.

Now observe that $s_2$ is a $\nearr$-maximal cell of $R'$, where $R'$ is the canonical sub-RSST of $R$ with reading word $\e{v}$.
There is no $\nwarr$ arrow of $R$ with tail $s_2$ because this would force an arrow square in  $R$ of the form on the right of \eqref{ep squares are strict}, which we are assuming does not exist.
This implies by Lemma~\ref{l small letter end} that we can assume $\e{v}$ ends in $\e{R_{s_2}v'}$ where $\e{v'}$ has letters $> R_{s_2} = a+1$.
Hence $\e{R_{s_2}v'R_{s_1}w} \equiv 0$ by Corollary~\ref{c temperley lieb} with  repeated letter $\e{a+1}$, implying that the square respecting reading word $\e{vR_{s_1}w}$ of  $R$ is 0 in  $\U/\Ilam{3}$.
\end{proof}

\section{Positive monomial expansion of noncommutative Schur functions}\label{s positive monomial}
After some preliminary definitions, we recall the main theorem and use it to give an explicit positive combinatorial formula for new variant $q$-Littlewood-Richardson coefficients indexed by a 3-tuple of skew shapes. In \textsection\ref{ss proof of theorem}, we prove a stronger, more technical version of the main theorem.
\subsection{Noncommutative flagged Schur functions}\label{ss noncommutative flagged schur}
Here we introduce the noncommutative Schur functions $\mathfrak{J}_{\lambda}(\mathbf{u})$ from \cite{FG,LamRibbon} and their flagged generalizations.
These generalizations will be put in a broader context in a future paper.

The \emph{noncommutative elementary symmetric functions} are given by
\[
e_d(S)=\sum_{\substack{i_1 > i_2 > \cdots > i_d \\ i_1,\dots,i_d \in S}}u_{i_1}u_{i_2}\cdots u_{i_d},
\]
for any subset $S$ of $\ZZ$ and positive integer $d$; set $e_0(S)=1$ and $e_{d}(S) = 0$ for $d<0$.
By \cite{LamRibbon}, $e_i(S) e_j(S) = e_j(S) e_i(S)$ in  $\U/\Ilam{k}$ for all $i$ and $j$.

Given a weak composition
$\alpha=(\alpha_1,\dots,\alpha_l)$ and sets $S_1,S_2,\dots, S_{l}\subseteq \ZZ$, define the \emph{noncommutative column-flagged Schur function} by
\begin{align}
J_{\alpha}(S_1,S_2,\dots, S_l)
:=\sum_{\pi\in \S_{l}}
\sgn(\pi) \, e_{\alpha_1+\pi(1)-1}(S_1) e_{\alpha_2+\pi(2)-2}(S_2)\cdots e_{\alpha_{l}+\pi(l)-l}(S_l). \label{e flag schur}
\end{align}
We also use the  following shorthands
\[J_\alpha^\mathbf{n} = J_\alpha(n_1,\ldots,n_l) = J_\alpha([n_1],\ldots,[n_l]),\]
where $\mathbf{n}=(n_1,\dots,n_l)$. These are related to the noncommutative Schur functions $\mathfrak{J}_{\lambda}(u_1,\dots,u_n)$ of \cite{FG} by
$\mathfrak{J}_{\lambda}(u_1,\dots,u_n) = J_{\lambda'}^{n\,n\,\cdots\,n}.$  When the $u_i$ commute, $\mathfrak{J}_{\lambda}(u_1,\dots,u_n)$ becomes the ordinary Schur function  $s_\lambda$ in $n$ variables, and  $J_{\lambda}^\mathbf{n}$ becomes the column-flagged Schur function $S_\lambda^*(\mathbf{1},\mathbf{n})$ studied in \cite{Wachs}, where  $\mathbf{1}$ denotes the all-ones vector of length  $l$.
%??this was double checked.

For words $\e{w^1}, \ldots, \e{w^{l-1}} \in \U$, we will also make use of the \emph{augmented noncommutative column-flagged Schur functions}, given by
\begin{align*}
&J_{\alpha}(S_1\Jnot{w^1}S_2\Jnot{w^2}\cdots\Jnot{w^{l-1}}S_l) \\
&:=\sum_{\pi\in \S_{l}}
\sgn(\pi) \, e_{\alpha_1+\pi(1)-1}(S_1)\e{w^1} e_{\alpha_2+\pi(2)-2}(S_2)\e{w^2}\cdots  \e{w^{l-1}}e_{\alpha_{l}+\pi(l)-l}(S_l).
\end{align*}
If $\e{w^i}$ is empty for all $i \in [l]$,  $i \ne j$, and  $S_i = [n_i]$, we also use the shorthand
\[ J_\alpha^\mathbf{n}(\Jnotb{j}{w}) = J_{\alpha}(S_1\Jnot{w^1}S_2\Jnot{w^2}\cdots\Jnot{w^{l-1}}S_l).
\]

The (augmented) noncommutative column-flagged Schur functions will be considered here as elements of $\U/\Ilam{k}$, unless stated otherwise.
%(These are interesting in other quotients of $\U$ in which the elementary symmetric functions commute, a direction which is pursued in \cite{BF} and discussed briefly in Section \ref{s generalizations}.)
Note that  because the noncommutative elementary symmetric functions commute in $\U/\Ilam{k}$,
\begin{align}\label{e elem sym Jswap}
J_\alpha^\mathbf{n} \equiv -J_{\alpha_1, \ldots, \alpha_{j-1},\alpha_{j+1}-1,\alpha_j+1,\ldots,\alpha_l}^\mathbf{n} \quad \text{whenever  $n_j = n_{j+1}$}.
\end{align}
In particular,
\begin{align}\label{e elem sym J0}
J_\alpha^\mathbf{n} \equiv 0 \quad \text{whenever  $\alpha_j = \alpha_{j+1}-1$ and  $n_j=n_{j+1}$.} \tag{$\diamond$}
\end{align}
More generally, \eqref{e elem sym Jswap} and \eqref{e elem sym J0} hold for the augmented case provided $\e{w^j}$ is empty and the assumption $n_j = n_{j+1}$ is replaced by $S_j = S_{j+1}$.

We will make frequent use of the following fact (interpret $[0] = \{\}$):
\begin{align}\label{e ek induction}
e_d([m]) =&~\e{m} e_{d-1}([m-1])+e_d([m-1]) & \text{if  $m > 0$ and  $d$ is any integer}.
\end{align}
Note that
\begin{align*}
e_d([0]) =&
\begin{cases}
  1 & \text{ if } d = 0, \\
  0 & \text{otherwise.}
\end{cases} \notag
\end{align*}
We will often apply this to  $J_\alpha^\mathbf{n}$ and its variants by expanding
$e_{\alpha_j+\pi(j)-j}(S_j)$ in \eqref{e flag schur} using \eqref{e ek induction} (so that \eqref{e ek induction} is applied once to each of the  $l!$ terms in the sum in \eqref{e flag schur}).
We refer to this as a \emph{$j$-expansion of  $J_\alpha^\mathbf{n}$} or simply a  \emph{$j$-expansion}.

%\begin{proposition}\label{p standard vs distinct}
%If  $\phi$ is an increasing map  $\phi: [n] \to S$, for  $S \subseteq \ZZ_{\geq 1}$, then  $f \in \U$ is zero in $\U/ I_{\le 3}$, then  $\phi(f) = 0$ in  $\U/I_{\le 3}$.  Similarly,
%if  $f \in \U$ is 0 in any bijectivization of   $\U/I_{\le 3}$, then the same holds for $\phi(f)$.
%\end{proposition}

\subsection{The main theorem}
\label{ss qLR coefs}
%For a weak composition $\lambda=(\lambda_1,\dots, \lambda_l)$, the \emph{initial nonincreasing index} of $\lambda$ is the smallest $j$
%such that $\lambda_j\geq \lambda_{j+1}$, where we set $\lambda_{l+1}=\lambda_{l+2}=\cdots=0$.
Let $\text{RSST}_{\lambda}$ denote the set of restricted square strict tableaux of shape $\lambda$.  For a diagram  $\theta$ contained in columns  $1,\ldots,l$ and nonnegative integers $n_1,\ldots,n_l$, let  $\Tab_\theta^{n_1 n_2 \cdots n_l}$ denote the set of tableaux of shape  $\theta$ such that the entries in column  $c$ lie in $[n_c]$.
Define  $\mathfrak{J}_\lambda(\mathbf{u}) = J_{\lambda'}(\ZZ,\ldots,\ZZ)$.  Recall from Definition \ref{d square reading word} that  $\sqread(T)$ is a specially chosen reading word of  $T$.
We have the following generalization of Theorem \ref{t J intro} from the introduction to the noncommutative column-flagged Schur functions:
\begin{theorem}\label{t flag Schur k3}
If $\lambda$ is a partition with  $l = \lambda_1$ columns and $0\leq n_1\leq n_2\leq \cdots\leq n_l$, then in the algebra $\U/\Ilam{3}$,
\begin{align*}%\label{et main thm1}
J_{\lambda'}^{n_1 n_2 \cdots n_l} =
\sum_{ T \in \RSST_{\lambda}, \ T \in \Tab_{\lambda}^{n_1 n_2 \cdots n_l}} \sqread(T).
\end{align*}
\end{theorem}
In the next subsection  we will prove a stronger, more technical version  of this theorem.    Recall that Theorem \ref{t J intro} states that
\begin{align*}
\qquad \ \ \ \ \ \mathfrak{J}_\lambda(\mathbf{u}) = \sum_{T \in \RSST_\lambda} \sqread(T) \qquad \ \text{in }\, \U/\Ilam{3}.
\end{align*}
This follows from Theorem \ref{t flag Schur k3} since  $J_{\lambda'}(\ZZ,\ldots,\ZZ)$ can be written as a sum of words in  $\U$, grouped according to the multiset of letters appearing in each word,  and each group can be computed using Theorem \ref{t flag Schur k3}.

We now use this to deduce an explicit combinatorial formula for the coefficients of the  Schur expansion of new variant combinatorial LLT polynomials indexed by a 3-tuple of skew shapes. This is a straightforward application of the Fomin-Greene machinery \cite{FG} (see also \cite{LamRibbon, BF}).

Define the $\field$-linear map
\[\Delta: \U\to \field[x_1, x_2, \dots] \ \text{ by }  \ \e{v} \mapsto Q_{\Des(\e{v})}(\mathbf{x}).\]
Let $\langle\cdot ,\cdot  \rangle$ be the symmetric bilinear form on $\U$ in which the monomials form an orthonormal basis.
Note that any element of $\U/\Ilam{k}$ has a well-defined pairing with any element of $(\Ilam{k})^\perp$.
We need the following variant of Theorem 1.2  of \cite{FG} and results in Section 6 of \cite{LamRibbon}; see \cite{BF} for a detailed proof.

\begin{theorem} \label{t basics}
For any $f\in (\Ilam{k})^\perp$,
\[\Delta(f) =
\Big\langle \sum_{\e{v}\in \U} Q_{\Des(\e{v})}(\mathbf{x})\e{v} , f \Big\rangle
 = \sum_{\lambda} s_\lambda(\mathbf x) \langle \mathfrak J_{\lambda}(\mathbf{u}) , f \rangle. \]
%Hence the coefficient of  $s_\lambda(\mathbf{x})$ in  $\Delta(f)$ is $\langle \mathfrak{J}_{\lambda}(\mathbf u) , f \rangle$.
\end{theorem}

Recall that the \emph{new variant $q$-Littlewood-Richardson coefficients}
 $\mathfrak{c}_{\bm{\beta}}^\lambda(q)$ are the coefficients in the Schur expansion of the new variant
combinatorial LLT polynomials, i.e.
\[\mathcal{G}_{\bm{\beta}}(\mathbf{x};q) = \sum_{\lambda}\mathfrak{c}_{\bm{\beta}}^\lambda(q)s_\lambda(\mathbf{x}).\]
The new variant $q$-Littlewood-Richardson coefficients are known to be polynomials in $q$ with nonnegative integer coefficients \cite{GH}, but positive combinatorial formulae for these coefficients have only been given in the cases that $k \le 2$ and the diameter of  $\bm{\beta}$ is  $\le 3$ (see \eqref{e diameter}). Below is the first positive combinatorial interpretation in the $k=3$ case.

In order to give an expression for the $\mathfrak{c}_{\bm{\beta}}^\lambda(q)$ that does not depend on arbitrary choices of square respecting reading words, we adapt the statistics $\Desi{3}$ and  $\invi{3}$ to RSST:
\begin{align*}
\Desi{3}(T) &= \{(T_z, T_{z'}) \mid z,z' \in \sh(T), \ T_z - T_{z'} = 3, \\
&\qquad\qquad\qquad\text{ and } (z <_\nearrsub z' \text{ or there is a  $\nwarr$ arrow from  $z$ to  $z'$ in  $T$)} \}, \\
\invi{3}(T) &= |\{(z, z') \mid z,z' \in \sh(T), \ 0 < T_z - T_{z'} < 3, \text{ and } z <_\nearrsub z'\}|,
\end{align*}
where  $T$ is any RSST and the first expression is a multiset.
We refer to these as the \emph{3-descent multiset} and the \emph{3-inversion number} of $T$.
One checks easily that
\begin{align}\label{e des tab}
\text{$\Desi{3}(T) = \Desi{3}(\e{v})$ and $\invi{k}(T) = \invi{k}(\e{v})$}
\end{align}
for any square respecting reading word  $\e{v}$ of  $T$.
Recall that  an RSST  $T$ is \emph{nonzero} if any (equivalently, every) square respecting reading word of  $T$ is nonzero in  $\U/\Ilam{3}$.
%\item $\Wi{k}(c,D') = $ the set of nonzero $k$-words $\e{v}$ such that sorting $\e{v}$ in weakly increasing order yields $c$ and $\Desi{k}(\e{v}) = D'$.

Recall from Definition \ref{d Wi} that $\Wi{k}(\bm{\beta})  = \{\e{v}\in \U \mid \bm{\delta} \circ_\mathbf{0}\e{v}=  \bm{\gamma}\}$, where  $\bm{\beta} = \bm{\gamma}/\bm{\delta}$ and $\circ_{\mathbf{0}}$ is an action of $\U$ on $k$-tuples of partitions.

\begin{corollary}\label{c main}
Let $\bm{\beta}$ be a 3-tuple of skew shapes with contents, and let $c,D'$ be the corresponding content vector
and multiset from Proposition~\ref{p equivalence classes} (ii).
Then the new variant $q$-Littlewood-Richardson coefficients are given by
\[\mathfrak{c}_{\bm{\beta}}^\lambda(q) = \sum_{\substack{\e{v}\in\Wi{3}(\bm{\beta}) \\\e{v}\in \{\sqread(T) \mid T \in \RSST_\lambda\}}}q^{\invi{3}(\e{v})}
= \sum_{\substack{T \in \RSST_\lambda, \ T \text{ nonzero} \\ \Desi{3}(T) = D', \, \text{$c =$ sorted entries of  $T$}} } q^{\invi{3}(T)}. \]
\end{corollary}

The advantage of the second expression is that it makes it clear that the answer depends only on a set of RSST and not on choices for their square respecting reading words.  In the special case that  $c$ contains no repeated letter, this second expression is equal to
\[\sum_{S \in \SYT_\lambda, \, \Desip{3}(c,S) = D'} q^{\invip{3}(c,S)}, \]
where  $\Desip{3}$ and  $\invip{3}$ are defined so that  $\Desip{3}(c,T^\stand) = \Desi{3}(T)$ and  $\invip{3}(c,T^\stand) = \invi{3}(T)$ whenever  $c$ is equal to the sorted entries of  $T$; here,  $T^\stand$ denotes the SYT obtained from  $T$ by replacing its smallest entry with 1, its next smallest entry with 2, etc.
This reformulation is essentially Haglund's conjecture \cite[Conjecture 3]{Haglund}. (Haglund's conjecture is for those new variant combinatorial LLT polynomials which appear in the expression for transformed Macdonald polynomials as a positive sum of LLT polynomials \cite{HHL}.  In particular, all such LLT polynomials correspond to a tuple of skew shapes that have a content vector with no repeated letter.)

\begin{proof}[Proof of Corollary \ref{c main}]
If $\e{v}= \e{v_1} \cdots \e{v_t} \in\U$ and  $\e{w} = \e{v_1} \cdots \e{v_{i-1} v_{i+1} v_i v_{i+2}} \cdots \e{v_t}$ with  $0 < \e{v_{i+1}} - \e{v_i} < 3$,
%i.e.  $\e{w}$ is obtained from  $\e{v}$ by applying the relation \eqref{e q commute},
then $\e{v}\equiv q^{-1} \e{w}$ by \eqref{e q commute} and $q^{-\invi{3}(\e{v})}\e{v}- q^{-\invi{3}(\e{w})}\e{w} \equiv 0$.
It follows from this and Proposition~\ref{p equivalence classes} (i)  that
\[\sum_{\e{v}\in\Wi{3}(\bm{\beta})} q^{\invi{3}(\e{v})}\e{v}\in (\Ilam{3})^\perp. \]

By definition of  $\Delta$ and Proposition \ref{p assafi},
\[\Delta \bigg( \sum_{\e{v}\in\Wi{3}(\bm{\beta})} q^{\invi{3}(\e{v})}\e{v}\bigg) =  \sum_{\e{v}\in\Wi{3}(\bm{\beta})}q^{\invi{3}(\e{v})}Q_{\Des(\e{v})}(\mathbf{x}) = \mathcal{G}_{\bm{\beta}}(\mathbf{x};q).
\]
By Theorem \ref{t basics}, the coefficient of  $s_\lambda(\mathbf{x})$ in $\mathcal{G}_{\bm{\beta}}(\mathbf{x};q)$ is
\[\mathfrak{c}_{\bm{\beta}}^\lambda(q) = \left\langle \mathfrak{J}_{\lambda}(\mathbf{u}), \sum_{\e{v}\in\Wi{3}(\bm{\beta})}q^{\invi{3}(\e{v})}\e{v} \right\rangle =
\sum_{\substack{\e{v}\in\Wi{3}(\bm{\beta}) \\\e{v}\in \{\sqread(T) \mid T \in \RSST_\lambda\}}}q^{\invi{3}(\e{v})},\]
where the second equality is Theorem \ref{t J intro}.
The second expression for  $\mathfrak{c}_{\bm{\beta}}^\lambda(q) $ follows from this one by Proposition~\ref{p equivalence classes} (ii) and \eqref{e des tab}.
\end{proof}

\begin{example}\label{ex LLT2}
Let
%$\bm{\beta} = ((2)/(1),(3,3)/(1,1),(3,3)/(2,1))$
$\bm{\beta} = (2/1,33/11,33/21)$, $\e{v}= \e{48714235}$, and  $c = 12344578$ be as in Example \ref{ex LLT}.
Let  $D' = \Desi{3}(\e{v}) =  \{(4,1),(7,4),(8,5)\}$.  The coefficient  $\mathfrak{c}_{\bm{\beta}}^{(4,3,1)}(q)$ is computed by finding all
the nonzero RSST  $T$ of shape  $(4,3,1)$ with 3-descent multiset  $D'$.  These are shown below, along with their square respecting reading words and
3-inversion numbers.

\[
\begin{array}{l|ccc}
\ \ T&
{ \tiny \tableau{
1&3&4&5\\2&7&8\\4\\}} &
{ \tiny \tableau{
1&2&4&5\\3&7&8\\4\\}} &
{ \tiny \tableau{
1&2&4&5\\3&4&7\\8\\}}
\\[5.4mm]
\sqread(T) &
\e{42173845} & \e{ 43172845}& \e{83412745} \\[1.4mm]
\invi{3}(T) &
4 & 5 & 5
\end{array}
\]
Hence
\[ \mathfrak{c}_{\bm{\beta}}^{(4,3,1)}(q) = q^4 + 2q^5. \]
%[ 4, 3, 1 ],
%{* 4, 5^^2 *},
%83412745 43172845 42173845 ,
%true,
%[ {{ \tiny \tableau{
%1&2&4&5\\3&4&7\\8\\}}\atop 8\ 3\ 4\ 1\ 2\ 7\ 4\ 5\ }, {{ \tiny \tableau{
%1&2&4&5\\3&7&8\\4\\}}\atop 4\ 3\ 1\ 7\ 2\ 8\ 4\ 5\ }, {{ \tiny \tableau{
%1&3&4&5\\2&7&8\\4\\}}\atop 4\ 2\ 1\ 7\ 3\ 8\ 4\ 5\ } ],
\end{example}

\subsection{Proof of Theorem \ref{t flag Schur k3}}\label{ss proof of theorem}

% no longer needed, incorporated in main proof
%\begin{lemma}\label{l semistandard 0}
%Suppose $\alpha$ is a weak composition, $\mathbf{n}=(n_1,\ldots,n_l)$ with  $0 \le n_1 \leq \cdots \leq n_{j-1} \leq n_j-3$, and  $\e{v}= v^L \e{n_j}v^R$, where either $v^R$ does not contain  $n_j+3$ or does not contain $n_j-3$. Then in  $\U/\Ilam{3}$,
%\[\e{v}J_\alpha(n_1, n_{j-1},n_j \Jnot{w} n_{j+1},\ldots,n_l) =\e{v}J_\alpha(n_1, n_{j-1},n_j+1 \Jnot{w} n_{j+1},\ldots,n_l).\]
%More generally this holds if there are arbitrary words after  $j$.
%\end{lemma}

After three preliminary results, we state and prove a more technical version of Theorem~\ref{t flag Schur k3}, which involves computing $J_\lambda^\mathbf{n}$ inductively by pealing off diagonals from the shape $\lambda'$. Recall that $\equiv$ denotes equality in $\U/\Ilam{k}$.

The following lemma  is key to the proof of Theorem \ref{t flag Schur k3}.
\begin{lemma}\label{l lambda1 eq lambda2 commute}
If  $m < x$, then in $\U/\Ilam{k}$,
\[
J_{(a,a)}([m]  \Jnot{x}   [m]) = \e{x} J_{(a,a)}([m], [m]).
\]
More generally,
if  $m < x$ and  $\alpha$ is a weak composition satisfying  $\alpha_j = \alpha_{j+1}$ and  $\mathbf{n} = (n_1,\ldots, n_l)$ with  $n_j = n_{j+1}$, then in $\U/\Ilam{k}$,
\[
J_\alpha^\mathbf{n}(\Jnotb{j}{x}) = J_\alpha^\mathbf{n}(\Jnotb{j-1}{x}).
\]
\end{lemma}
\begin{proof}
Since the proofs of both statements are essentially the same, we prove only the first to avoid extra notation. We compute
\begin{align*}
0& \equiv J_{(a, a+1)}(\{x\}\cup[m], \{x\}\cup[m])\\
&=\e{x}J_{(a-1, a+1)}([m],\{x\}\cup[m])+
J_{(a, a+1)}([m], \{x\}\cup[m])\\
&=\e{x}J_{(a-1, a)}([m]  \Jnot{x}  [m])+
\e{x}J_{(a-1, a+1)}([m],[m])+
J_{(a, a)}([m]   \Jnot{x}  [m])+
J_{(a, a+1)}([m], [m])\\
& \equiv
\e{x}J_{(a-1, a+1)}([m],[m])+
J_{(a, a)}([m]   \Jnot{x}  [m]) \\
&\equiv -\e{x}J_{(a, a)}([m],[m])+
J_{(a, a)}([m]   \Jnot{x}  [m]).
\end{align*}
The first congruence is by \eqref{e elem sym J0}.  The equalities are a 1-expansion and a 2-expansion (see \textsection\ref{ss noncommutative flagged schur} for notation). The second congruence is by Corollary~\ref{c temperley lieb} with repeated letter $\e{x}$ and \eqref{e elem sym J0}. The last congruence is by \eqref{e elem sym Jswap}.
\end{proof}

%\begin{lemma}\label{l lambda1 eq lambda2 ss}
%Let $k$ be a positive integer. Suppose $\alpha$ is a weak composition with  $l$ parts and  $\alpha_j = \alpha_{j+1}$.
%Suppose  $\mathbf{n} = (n_1,\ldots, n_l)$ such that $n_1 \le n_2 \le \cdots \le n_l$ and  $n_{j+1} = n_{j}+1$.
%Suppose $x > n_{j+1}$ and  $x \ne n_{j+1} + k$.
%Let  $\e{v}=\e{v}^L n_{j+1} v^R$ such that  $v^R$ does not contain $n_{j+1}+k$.
%Then in  the algebra $\U/\Ilam{k}$,
%\begin{equation}\label{el lambda1 eq lambda2 ss}
%\e{v}J_\alpha^\mathbf{n}(\Jnotb{j}{x}) =\e{v}J_\alpha^\mathbf{n}(\Jnotb{j-1}{x}).
%\end{equation}
%\end{lemma}

The next lemma is a slight improvement on Lemma \ref{l lambda1 eq lambda2 commute} needed to handle the case of repeated letters.
\begin{lemma}\label{l lambda1 eq lambda2 ss}
Let $k$ be a positive integer. Suppose
\begin{list}{\emph{(\alph{ctr})}}{\usecounter{ctr} \setlength{\itemsep}{2pt} \setlength{\topsep}{3pt}}
\item $\alpha$ is a weak composition with  $l$ parts and  $\alpha_j = \alpha_{j+1}$,
\item $\mathbf{n} = (n_1,\ldots, n_l)$ satisfies $0 \le n_1 \le n_2 \le \cdots \le n_l$, and $n_{j-1} < n_{j+1}-3$, and  $n_{j+1} = n_{j}+1$,
\item $\e{x}$ is a letter such that $x > n_{j+1}$ and  $x \ne n_{j+1} + k$,
\item $\e{v}\in \U$ is a word such that  $\e{vn_{j+1}} \equiv 0$.
\end{list}
Then in  the algebra $\U/\Ilam{k}$,
\begin{equation}\label{el lambda1 eq lambda2 ss}
\e{v}J_\alpha^\mathbf{n}(\Jnotb{j}{x}) =\e{v}J_\alpha^\mathbf{n}(\Jnotb{j-1}{x}).
\end{equation}
\end{lemma}
\begin{proof}
Set $\alpha_- = (\alpha_1,\dots,\alpha_{j-1},\alpha_j-1,\alpha_{j+1},\dots)$ and $\mathbf{n}_+ = (n_1,\ldots, n_{j-1}, n_{j+1}, n_{j+1}, n_{j+2},\ldots)$.
By Lemma \ref{l lambda1 eq lambda2 commute},
\[
\e{v}J_\alpha^{\mathbf{n}_+}(\Jnotb{j}{x}) \equiv\e{v}J_\alpha^{\mathbf{n}_+}(\Jnotb{j-1}{x}).
\]
Apply a $j$-expansion to both sides to obtain
\begin{align}
&\e{v}J_{\alpha_-}(n_1, \ldots, n_{j-1}  \Jnot{n_{j+1}}   n_{j} \Jnot{x} n_{j+1},\ldots) +\e{v}J_\alpha^\mathbf{n}(\Jnotb{j}{x}) \label{e extra letter trick1}\\
\equiv~&\e{v}J_{\alpha_-}(n_1, \ldots, n_{j-1}  \Jnot{xn_{j+1}}   n_{j},  n_{j+1},\ldots) +\e{v}J_\alpha^\mathbf{n}(\Jnotb{j-1}{x}). \label{e extra letter trick2}
\end{align}
It suffices to show that the first term of \eqref{e extra letter trick1} and the first term of \eqref{e extra letter trick2} are $\equiv0$. We only show the latter, as the former is similar and easier.
There holds
\begin{align*}
&\e{v}J_{\alpha_-}(n_1, \ldots, n_{j-1}  \Jnot{xn_{j+1}}   n_{j},  n_{j+1},\ldots)
\equiv q^a \e{v}J_{\alpha_-}(n_1, \ldots, n_{j-1}  \Jnot{n_{j+1}x}   n_{j},  n_{j+1},\ldots) \\
&\equiv q^a \e{v n_{j+1}} J_{\alpha_-}^\mathbf{n}(\Jnotb{j-1}{x}) \equiv 0,
\end{align*}
where  $a \in \{0,1\}$.
The first congruence is by (c) and the relations \eqref{e q commute} and the far commutation relations \eqref{e far commute}, the second congruence is by (b) and the far commutation relations, and the last congruence is by (d).
%Set $\alpha_- = (\alpha_1,\dots,\alpha_j,\alpha_{j+1}-1,\alpha_{j+2},\dots)$.
%By applying a $\alpha_{j+1}$-expansion to both sides of \eqref{el lambda1 eq lambda2 ss}, we see that the desired result is equivalent to
%\begin{align*}
%&~\e{v}J_{\alpha_-}(n_1, \ldots,n_j   \Jnot{xn_{j+1}}   n_{j+1}-1,\ldots) +\e{v}J_{\alpha}(n_1, \ldots,n_j   \Jnot{x}   n_{j+1}-1,\ldots) \\
%\equiv &~ \scalebox{.95}[1.0]{$\e{v}J_{\alpha_-}(n_1, \ldots,n_{j-1}  \Jnot{x}   n_j   \Jnot{n_{j+1}}   n_{j+1}-1,\ldots) +\e{v}J_{\alpha}(n_1, \ldots,n_{j-1}   \Jnot{x}  n_j, n_{j+1}-1,\ldots)$}.
%\end{align*}
%This holds by Corollary \ref{c temperley lieb} with repeated letter  $n_{j+1}$ and Lemma \ref{l lambda1 eq lambda2 commute}.
\end{proof}

\begin{corollary}\label{c aaa+1 commute lam}
Maintain the notation of Lemma~\ref{l lambda1 eq lambda2 ss} and assume (a)--(d) of the lemma hold. Assume in addition
\begin{list}{\emph{(\alph{ctr})}}{\usecounter{ctr} \setlength{\itemsep}{2pt} \setlength{\topsep}{3pt}}
\item[\rm (e)] $\e{w} \in \U$ such that $\e{w} = \e{x w^R}$ and the letters of $\e{w^R}$ are $> n_{j+1}+k$.
\end{list}
Then in the algebra $\U/\Ilam{k}$,
\begin{equation*}
\e{v}J_{\alpha}(n_1,\dots, n_j  \Jnot{w}   n_{j+1},\dots) = \e{v}J_{\alpha}(n_1,\dots,  n_{j-1}   \Jnot{x}   n_j  \Jnot{w^R}   n_{j+1},\dots). \\
\end{equation*}
\end{corollary}
\begin{proof}
We compute in  $\U/\Ilam{k}$,
\begin{align*}
& \e{v}J_{\alpha}(n_1,\dots, n_j  \Jnot{w}   n_{j+1},\dots) \\
&= \e{v}J_{\alpha}(n_1,\dots,  n_{j-1}, n_j   \Jnot{x}   n_{j+1}   \Jnot{w^R}   \dots) \\
&= \e{v}J_{\alpha}(n_1,\dots,  n_{j-1}   \Jnot{x}   n_j, n_{j+1}   \Jnot{w^R}   \dots) \\
&= \e{v}J_{\alpha}(n_1,\dots,  n_{j-1}   \Jnot{x}   n_j   \Jnot{w^R}   n_{j+1},\dots).
\end{align*}
The first and third equalities are by the far commutation relations and the second equality is by Lemma~\ref{l lambda1 eq lambda2 ss} (the proof of the lemma still works with the word  $\e{w^R}$ present).
\end{proof}

For a weak composition or partition $\alpha=(\alpha_1,\dots,\alpha_l)$, define $\alpha_{l+1} = 0$.
\begin{theorem}\label{t flag Schur k3 technical}
Let  $\lambda' \setminus \alpha'$ be a restricted shape and let  $l$ be the number of parts of  $\lambda$.
%Set $j_0 = \max(\{i \in [l] \mid \alpha_i = 0\}\cup\{1\})$,
Set $j = \min(\{i \mid \alpha_i > 0, \ \alpha_i \ge \alpha_{i+1}\} \cup \{l+1\} )$ and $j' = \max(\{i \mid \alpha_i < \lambda_i\} \cup \{0\})$  (see Figure~\ref{f main proof} and the discussion following Remark \ref{r main theorem}).
Suppose
\begin{list}{\emph{(\roman{ctr})}}{\usecounter{ctr} \setlength{\itemsep}{2pt} \setlength{\topsep}{3pt}}
\item  $\alpha$ is of the form
\begin{align*}
&(0,\dots,0,1,2,\dots,a,a,a+1,a+2,\dots,\alpha_{j'}, \lambda_{j'+1},\lambda_{j'+2},\dots), \quad \text{ or } \\
&(0,\dots,0,1,2,\dots,\alpha_{j'}, \lambda_{j'+1},\lambda_{j'+2},\dots),
\end{align*}
where $\alpha_{j'} \ge \alpha_{j'+1}-1 = \lambda_{j'+1}-1$,
the initial run of 0's may be empty, and
on the top line (resp. bottom line) $j$ is  $< j'$ and is the position of the first  $a$ (resp.  $j$ is  $j'$ or  $j'+1$).
%\item $\alpha_i = \alpha_{i+1}- 1$ for  $i \in \{j_0,\dots,j'-1\} \setminus \{j\}$, and if  $j < j'$ then $\alpha_j = \alpha_{j+1}$,
\item $R$ is an RSST of shape  $\lambda' \setminus \alpha'$ with entries $r_1 = R_{\alpha_1+1,1}, \ldots, r_{j'} = R_{\alpha_{j'}+1,j'}$ along its northern border,
\item $\e{vw}$ is a square respecting reading word of $R$ such that $\e{w}$ is a subsequence of $\e{r_{j+1} \cdots r_{j'}}$ which contains $\e{r_{j+1}}$ if $r_{j+1}>r_j+1$,
\item $\mathbf{n} =(n_1,\ldots, n_l)$ is a tuple which satisfies $0 \leq n_1 \leq n_2 \leq \cdots \leq n_l$, and
\item $n_c = r_c-1$ for  $c \in [j']$ with the exception that  $n_j$ may be  $< r_j -1$ if  $\e{w}$ is empty or $w_1 \ne r_j+1$.
\end{list}
Then in the algebra $\U/\Ilam{3}$,
\begin{equation}\label{e new statement}
\e{v}J_\alpha^{\mathbf{n}}(\Jnotb{j}{w})
 = \displaystyle \sum_{\substack{T \in \RSST_{\lambda'},
\ T_{\lambda' \setminus \alpha'} = R \\ T_{\alpha'} \in \Tab_{\alpha'}^{\mathbf{n}}} } \sqread(T).
\end{equation}
\end{theorem}

To parse this statement, it is  instructive to first understand the case when $j\ge j'$
(which implies $\e{w}$ is empty) and $\alpha_2 > 0$.  In this case the theorem becomes
\[
\begin{array}{c}
\parbox{14.6cm}{
Suppose $\lambda' \setminus \alpha'$ is a restricted shape with $\alpha_1 = \alpha_2-1 = \cdots = \alpha_{j}-j+1$ and $\alpha_j \geq \alpha_{j+1} \geq \cdots \geq \alpha_{l} > \alpha_{l+1} = 0$.
Let  $R$ be an  RSST of shape  $\lambda' \setminus \alpha'$ with entries $r_1, \ldots, r_{j'}$ along its northern border.
Suppose $\mathbf{n} = (n_1,\ldots, n_l)$ satisfies $0 \leq n_1 \leq n_2 \leq \cdots \leq n_l$,
$n_c = r_c-1$ for  $c \in [j-1]$, and  $n_j < r_j$.
Then in $\U/\Ilam{3}$,
}
\\ \\[-2mm]
\sqread(R) J_\alpha^{\mathbf{n}}
 = \displaystyle \sum_{\substack{T \in \RSST_{\lambda'},
\ T_{\lambda' \setminus \alpha'} = R \\ T_{\alpha'} \in \Tab_{\alpha'}^{\mathbf{n}}} } \sqread(T).
\end{array}
\]

\begin{remark}
\label{r main theorem}
\
\begin{list}{\rm{(\alph{ctr})}}{\usecounter{ctr} \setlength{\itemsep}{2pt} \setlength{\topsep}{3pt}}
\item Theorem~\ref{t flag Schur k3} is the special case of Theorem~\ref{t flag Schur k3 technical} when $\alpha = \lambda$ (the  $\lambda$ and  $\lambda'$ of Theorem \ref{t flag Schur k3} must be interchanged to match the notation here).
\item A reading word  $\e{vw}$ as in (iii) always exists---for instance take  $\e{vw} = \sqread(R)$ with  $\e{w}$ the subsequence of
$\e{r_{j+1} \cdots r_{j'}}$ consisting of those entries which are not the tail of a  $\nwarr$ arrow of  $R$.
\item For any  $\e{vw}$ satisfying the the assumptions of the theorem, $\e{w}$ does not contain the tail of a $\nwarr$ arrow of $R$.
\item It follows from the theorem that  $\e{v}J_\alpha^{\mathbf{n}}(\Jnotb{j}{w}) = 0$ if  $R$ cannot be completed to an RSST of shape $\lambda'$, however we purposely do not make this assumption so that it does not have to be verified at the inductive step.
%i\item For the corner cases  $\alpha= (0,\dots, 0)$, and  $\alpha = \lambda$, we have  $j = l+1, j' = l, \e{w}$ empty and $j = 1, j' = 0, \e{w}$ empty, respectively; these are essentially just conventions and do not play an important role in the proof
\end{list}
%$\alpha_{j'+1} \geq \alpha_{j'+2} \geq \cdots$.
\end{remark}

The assumptions on $j$, $j'$, and  $\alpha$ look complicated, but their purpose is simply to peal off the entries of a tableau of shape  $\lambda'$ one diagonal at a time, starting with the southwesternmost diagonal, and reading each diagonal in the $\nwarr$ direction (see Figure~\ref{f main proof}).
%Precisely,
% $\lambda' \setminus \alpha'$ is the union of the diagonals of  $\lambda$ with contents  $\leq j_0-1$ if  $\alpha_1 =0$ and contents  $ \le -\alpha_1$ if  $\alpha_1 > 0$, together with the partial diagonal $(\alpha_{j+1}+1,j+1), \ldots, (\alpha_{j'}+1,j')$
% (see Figure~\ref{f main proof}).
The proof goes by induction, pealing off one letter at a time from $J_\alpha^\mathbf{n}$, in the order just specified, and incorporating them into  $R$.  The index  $j$ indicates the column of the next letter to be removed.

The reader is encouraged to follow along the proof below with Example \ref{ex inductive tree}.
%?? The proof is substantially simpler if we quotient by words with repeated letter, so it may be helpful to assume this on first reading of the proof;

\begin{proof}[Proof of Theorem \ref{t flag Schur k3 technical}]
Write $J$ for $J_{\alpha}^{\mathbf{n}}(\Jnotb{j}{w})$.
The proof is by induction on $|\alpha|$ and the $n_i$.
If  $|\alpha|=0$,  $J_\alpha^\mathbf{n}$ is (a noncommutative version of) the determinant of an upper unitriangular matrix, hence  $J = \e{w}$. The theorem then
states that  $\e{vw} \equiv \sqread(R)$, which holds by Theorem \ref{t square respecting connected}.

Next consider the case $n_1= 0$ and  $\alpha_1>0$.  This implies that  $J_\alpha^\mathbf{n}$ is (a noncommutative version of)
the determinant of a matrix whose first row is all 0's, hence $J=0$.
The right side of \eqref{e new statement} is also 0 because $\Tab_{\alpha'}^\mathbf{n}$ is empty for $n_1= 0$ and  $\alpha_1>0$.

If $n_i=n_{i+1}$ for any $i \ne j$ such that  $\alpha_i = \alpha_{i+1}-1$, then $J=0$ by \eqref{e elem sym J0}. To see that the right side of \eqref{e new statement} is also 0 in this case, observe that if  $T $ is an RSST from this sum, then
\[n_i+1 = r_{i} < T_{\alpha_{i+1},i+1} \leq n_{i+1} = n_i,\]
hence the sum is empty.

By what has been said so far, we may assume  $|\alpha| > 0$, $\alpha_j > 0$, and $0 \le n_{j-1}<n_j$ if  $j > 1$ and  $0 < n_j$ if  $j =1$.  We proceed to the main body of the proof.
There are two cases depending on whether  ($\e{w}$ is empty or $w_1 > r_j +1$) or  $w_1 = r_j +1$.
%Note that in the latter case we must have  $w_1 = r_{j+1}$.
Set $\mathbf{n}_- = (n_1, n_2, \cdots, n_{j-1}, n_j-1,n_{j+1},\cdots,n_l)$ and  $\alpha_- = (\alpha_1,\dots, \alpha_{j-1}, \alpha_j-1, \alpha_{j+1},\dots)$.

\textbf{Case  $\e{w}$ is empty or $w_1 > r_j +1$.}

There are three subcases depending on whether $j=1$,  ($j>1$ and $\alpha_j = 1$), or  ($j > 1$ and  $\alpha_j > 1$).  We argue only the last, as the first two are similar and easier.
A $j$-expansion yields
\begin{align}
\e{v}J=~&
\e{v}J_{\alpha_-}(n_1,\dots, n_{j-1}   \Jnot{n_j}   n_j-1 \Jnot{w} n_{j+1},\dots)
+\e{v}J_{\alpha}^{\mathbf{n}_-}(\Jnotb{j}{w})\label{e case no square-1}\\
\equiv~&
\e{v}J_{\alpha_-}(n_1,\dots, n_{j-1} \Jnot{n_jw} n_j-1,n_{j+1},\dots)
+\e{v}J_{\alpha}^{\mathbf{n}_-}(\Jnotb{j}{w}) \label{e case no square0} \\
\equiv~&
 \sum_{\substack{ T \in \RSST_{\lambda'}, \ T_{\lambda' \setminus \alpha'_-} = R\, \sqcup \,{\tiny\tableau{\e{n_j}}}_{\alpha_j, j}\\ T_{\alpha'_-} \in \Tab_{\alpha'_-}^{\mathbf{n}_-}}} \sqread(T)
+  \sum_{\substack{T \in \RSST_{\lambda'}, \ T_{\lambda' \setminus \alpha'} = R \\ T_\alpha' \in \Tab_{\alpha'}^{\mathbf{n}_-}}} \sqread(T), \label{e case no square}
\end{align}
The first congruence is by the far commutation relations since the letters of  $\e{w}$ are  $> n_j +2$.
The second term of \eqref{e case no square0} is congruent ($\!\bmod{\, \Ilam{3}}$) to  the
second sum of \eqref{e case no square} by induction.
The conditions of the theorem are satisfied here: (i)--(iii) are clear,  (iv) holds since  $n_{j-1} < n_j$, and (v) holds as $\e{w}$ is empty or  $w_1 \ne r_j+1$ so it is okay that  $(n_-)_{j} < r_j-1$.

%If  $j=1$, then we claim that the first term of \eqref{e case no square0} satisfies the the conditions of the theorem (with  $R \sqcup {\tiny\tableau{\e{n_j}}}_{\alpha_j, j}$ in place of  $R$,  $\e{v}\e{n_1w}$ in place of  $\e{v}$,  $\e{w}$ empty,  $j$ the initial nonincreasing index of  $\alpha_-$), so it is equal to the first sum of \eqref{e case no square} by induction.
%Decoding notation, we have $\e{v}J_{\alpha_-}(n_1,\dots, n_{j-1} \Jnot{n_jw} n_j-1,n_{j+1},\dots) =\e{v}\e{n_1}\e{w} J_{\alpha_-}^{\mathbf{n}_-}$.
%The pseudo-partition  $\alpha_-$ satisfies the necessary conditions. The  $n < r$ condition is clear.
%Finally, the word $\e{n_1 w}$ is the end of a square respecting reading word of  $R \sqcup {\tiny\tableau{\e{n_1}}}_{\alpha_j, j}$ satisfying the required conditions: the danger here is that $R \sqcup {\tiny\tableau{\e{n_1}}}_{\alpha_j, j}$ has a $\nwarr$ arrow from  $r_{j+1}$ to  $n_j$ and  $r_{j+1} = w_1$, but this implies  $w_1 = r_j +1$, contradicting our assumption $w_1 > r_j+1$.
We next claim that either the first term of \eqref{e case no square0} satisfies conditions (i)--(v) of the theorem (with   $\alpha_-$ in place of  $\alpha$, $\mathbf{n}_-$ in place of  $\mathbf{n}$, $R \sqcup {\tiny\tableau{\e{n_j}}}_{\alpha_j, j}$ in place of  $R$,  $\e{n_jw}$ in place of  %??useful this needs modification for the \alpha_j=1 case
$\e{w}$,  $j-1$ in place of   $j$), or (ii) fails and the first term of \eqref{e case no square0} is 0.  This will show that the first term of \eqref{e case no square0} is congruent to the first sum of \eqref{e case no square} (by induction in the former case, and because
both quantities are  $\equiv$ 0 in the latter case).
%so it is equal to the first sum of \eqref{e case no square} by induction.
We now assume that (ii) holds, i.e. $R \sqcup {\tiny\tableau{\e{n_j}}}_{\alpha_j, j}$ is an RSST, and we verify the remaining conditions.
Conditions (i) and (v) are clear, and (iv) follows from $n_{j-1}< n_j$.
Next we check (iii), which requires showing that $\e{n_j w}$ is the end of a square respecting reading word of  $R \sqcup {\tiny\tableau{\e{n_j}}}_{\alpha_j, j}$ satisfying an extra
condition.  There are two ways this can fail: either (I) $R \sqcup {\tiny\tableau{\e{n_j}}}_{\alpha_j, j}$ has a $\searr$ arrow from  $n_j$ to  $r_{j+1}$ and  $w_1\ne r_{j+1}$, or (II) $R \sqcup {\tiny\tableau{\e{n_j}}}_{\alpha_j, j}$ has a $\nwarr$ arrow from  $r_{j+1}$ to  $n_j$ and  $r_{j+1} = w_1$; (I) implies  $r_{j+1} > r_j +1$ so by assumption (iii) of the theorem $w_1 = r_{j+1}$,  thus (I) cannot occur; (II) implies  $w_1 = r_j +1$, contradicting our assumption $w_1 > r_j+1$, thus (II) cannot occur.

%Note that the argument just given works in the case $j=l$ because in that case $\alpha=(a, a+1,\dots, a+l-1)$
%for $a\geq 0$, which can also be interpreted as the pseudo-partition $(a, a+1, \dots, a+l-1, 0)$.

It remains to show that if (ii) does not hold, i.e. $R \sqcup {\tiny\tableau{\e{n_j}}}_{\alpha_j, j}$ is not an RSST, then  the first term of \eqref{e case no square0} is  $\equiv$ 0.
%$\e{v}J_{\alpha_-}^{\mathbf{n}_-}(\Jnotb{j-1}{n_jw}) = 0$.
There are two ways $R \sqcup {\tiny\tableau{\e{n_j}}}_{\alpha_j, j}$ can fail to be an RSST:
(A)  $n_j \leq r_{j-1}$, %(and  $j >1$),
or (B)  $n_j = r_{j+1}-2$ (and  $j < j'$).

In the case (A) holds, we have $n_j \le r_{j-1} = n_{j-1}+1$; in fact, since $n_{j-1} < n_j$, there holds $n_{j-1}+1 = r_{j-1}= n_j$.
By \eqref{e case no square-1} and \eqref{e case no square0},
it suffices to show  $\e{v}J \equiv 0$ since
$\e{v}J_{\alpha}^{\mathbf{n}_-}(\Jnotb{j}{w}) \equiv 0$ by \eqref{e elem sym J0} as
$(n_-)_{j-1} = (n_-)_j$ and  $\alpha_{j-1} = \alpha_{j}-1$.  Let  $\hat{R}$ be the canonical sub-RSST of  $R$ with reading word  $\e{v}$.
Corollary \ref{c removable 0} with  RSST  $\hat{R}$ and  $z$ the cell containing  $r_{j-1}$ implies  $\e{\e{v}r_{j-1}} \equiv 0$, hence
\begin{equation}\label{e caseA}
0 \equiv \e{v r_{j-1}} J_{\alpha-\epsilon_{j-1}}^\mathbf{n}(\Jnotb{j}{w})  \equiv\e{v}J_{\alpha-\epsilon_{j-1}}(n_1, \ldots, n_{j-2}  \Jnot{r_{j-1}}   n_{j-1},n_j \Jnot{w} \ldots),
\end{equation}
where the last congruence is by the far commutation relations since  $n_{j-2} = r_{j-2}-1 \le r_{j-1}-4$;
%or is trivial if j=2
here and throughout the proof, $\epsilon_i$ denotes the vector with a 1 in the  $i$-th position and 0's elsewhere.

Next, we compute
\[ 0 \equiv\e{v}J_\alpha(n_1, \ldots,n_{j-2},n_j,n_j \Jnot{w} \cdots) =\e{v}J_{\alpha-\epsilon_{j-1}}(n_1, \ldots, n_{j-2}  \Jnot{r_{j-1}}   n_{j-1},n_j \Jnot{w} \ldots)+\e{v}J \equiv \e{v}J.
\]
The first congruence uses \eqref{e elem sym J0},
the equality is a $j-1$-expansion, and the last congruence is by \eqref{e caseA}.

In the case that (B) holds, we must have  $n_j+1 = r_j = r_{j+1}-1$.  This, together with $r_{j+2} \geq r_{j+1}+3$, implies
\begin{equation}\label{e diag word no residue}
\text{$\e{r_{j+1} \cdots r_{j'}}$ does not contain  $r_j+3$ or  $r_j-3$},
\end{equation}
which will be important later.
% and hence  $r_{j+1} \ne w_1$; this means  $r_{j+1}$ belongs to  $\e{v}$.
%Let  $z = (\alpha_{j+1}+1,j+1)$ be the cell containing  $r_{j+1}$.
Apply a  $j+1$-expansion twice to the first term of \eqref{e case no square0} and for convenience set  $m = n_j$ (note that  $n_{j+1} = r_{j+1}-1 = n_j+1$):
\begin{align*}
&\e{v}J_{\alpha_-}(n_1,\dots, n_{j-1}  \Jnot{n_jw}  n_j-1,n_{j+1},\dots)\\
&=\e{v}J_{\beta}(n_1,\dots, n_{j-1}  \Jnot{n_jw}  n_j-1 \Jnot{n_{j+1} (n_{j+1}-1)} n_{j+1}-2,\dots) \\
&+\e{v}J_{\gamma}(n_1,\dots, n_{j-1}  \Jnot{n_jw}  n_j-1  \Jnot{n_{j+1} }   n_{j+1}-2,\dots) \\
&+\e{v}J_{\gamma}(n_1,\dots, n_{j-1}  \Jnot{n_jw}  n_j-1  \Jnot{n_{j+1}-1}   n_{j+1}-2,\dots) \\
&+\e{v}J_{\alpha_-}(n_1,\dots, n_{j-1}  \Jnot{n_jw}  n_j-1, n_{j+1}-2,\dots)\\
&=\e{v}J_{\beta}(n_1,\dots, n_{j-1}  \Jnot{mw}  m-1  \Jnot{r_j m}   m-1,\dots) \\
&+\e{v}J_{\gamma}(n_1,\dots, n_{j-1}  \Jnot{mw}  m-1  \Jnot{r_j}   m-1,\dots) \\
&+\e{v}J_{\gamma}(n_1,\dots, n_{j-1}  \Jnot{mw}  m-1  \Jnot{m}   m-1,\dots) \\
&+\e{v}J_{\alpha_-}(n_1,\dots, n_{j-1}  \Jnot{mw}  m-1, m-1,\dots),
\end{align*}
where  $\beta = \alpha_- - 2\epsilon_{j+1}$ and  $\gamma = \alpha_- - \epsilon_{j+1}$.
The first and third terms are  $\equiv 0$ by Corollary \ref{c temperley lieb} with  repeated letter $\e{m}$. The fourth term is  $\equiv 0$ by \eqref{e elem sym J0}.
The second term is equal to  $\e{v}J_{\gamma}(n_1,\dots, n_{j-1}   \Jnot{mwr_j}   m-1, m-1,\dots)$ by Lemma \ref{l lambda1 eq lambda2 commute}.  Since $n_{j-1} = r_{j-1}-1 \leq r_j-4$ and \eqref{e diag word no residue} holds, this is equal in  $\U/\Ilam{3}$ to an expression beginning with  $\e{vr_j}$ times some power of  $q$.

We claim that $\e{vr_j} \equiv 0$, as follows:
by Theorem \ref{t square respecting connected}, $\e{v}\equiv \e{v'} \e{w'}$, where  $\e{v'}\e{w'}$ is a square respecting reading word  of $R$ such that $\e{w'} \e{w}$ is a permutation of $r_{j+1} \cdots r_{j'}$.
%there is a square respecting reading word of  $R$ of the form $\e{v'} \e{w'} \e{w}$, where  $\e{w'} \e{w}$ is some permutation of $r_{j+1} \cdots r_{j'}$.
Let  $R'$ be the canonical sub-RSST of  $R$ with reading word $\e{v'}$ i.e. the RSST obtained from  $R$ by deleting its northeasternmost diagonal.
Now apply Corollary \ref{c removable 0} with RSST  $R'$ and  $z$ the cell containing  $\e{r_j}$ to obtain  $\e{v'r_j} \equiv 0$, hence
\[\e{vr_j} \equiv \e{v'}\e{w' r_j} \equiv q^a \e{v'r_j w'} \equiv 0, \]
where second congruence is by \eqref{e diag word no residue} ($a$ is some integer).
This completes the proof that the first term of \eqref{e case no square0} is  $\equiv 0$ when (B) holds.
This, in turn, completes the proof that the first term of \eqref{e case no square0} is congruent to the first sum of \eqref{e case no square}.

Now \eqref{e case no square} is equal to the right side of \eqref{e new statement}
because \eqref{e case no square} is simply the result of partitioning the set  $\{T \in \RSST_{\lambda'} \mid T_{\lambda' \setminus \alpha'} = R, \ T_{\alpha'} \in \Tab_{\alpha'}^{\mathbf{n}}\}$ into two, depending on whether or not  $T_{\alpha_j,j}$ is equal to or less than  $\e{n_j}$.  This completes the case  $\e{w}$ is empty or $w_1 > r_j+1$.

\textbf{Case $w_1 = r_j +1$.}

We have $w_1 = r_{j+1}$ and $n_j = r_j-1=w_1-2 = n_{j+1}-1$. Note that \eqref{e diag word no residue} holds here as well.
We now verify the hypotheses (a)--(e) of  Corollary~\ref{c aaa+1 commute lam}.
Conditions (a) and (b) are clear from what has just been said and from  $n_{j-1} < r_{j-1} \le r_j-3= n_{j+1}-3$.  Conditions (c) and (e) follow from  \eqref{e diag word no residue}.
Condition (d) holds because the same argument given two paragraphs above also applies here to give $\e{vn_{j+1}}= \e{vr_j} \equiv 0$.

Set  $t = |\e{w}|$. Corollary \ref{c aaa+1 commute lam} yields
\begin{align}
\e{v}J_{\alpha}^{\mathbf{n}}(\Jnotb{j}{w}) \notag
&\equiv \e{v}J_{\alpha}(n_1,\dots,  n_{j-1}   \Jnot{w_1}   n_j \Jnot{w_2 \cdots w_t} n_{j+1},\dots) \\
&\equiv \e{vw_1} J_{\alpha}^{\mathbf{n}}(\Jnotb{j}{w_2 \cdots w_t}). \label{e case square}
\end{align}
The last congruence is by the far commutation relations if  $j > 1$ since $n_{j-1}<r_{j-1}\le r_j-3 = w_1-4$ (if $j=1$ there is nothing to prove).

Finally, observe that the case  ($\e{w}$ is empty or  $w_1 > r_j+1$) applies to  \eqref{e case square} (with  $\e{vw_1}$ in place of  $\e{v}$,  $\e{w_2} \cdots \e{w_t}$ in place of $\e{w}$).
Since the right side of \eqref{e new statement} depends only on $R$ and not directly on  $\e{v}$ and  $\e{w}$, \eqref{e case square} is congruent to this right side,  and this gives the desired statement in the present case  $w_1 = r_j+1$.
%??rephrase? last part
%\[&\equiv \e{vw_1} J_{\alpha}^{\mathbf{n}}(\Jnotb{j}{w_2 \cdots w_t})=
% = \displaystyle \sum_{\substack{T \in \RSST_{\lambda'},
%\ T_{\lambda' \setminus \alpha'} = R \\ T_{\alpha'} \in \Tab_{\alpha'}^{\mathbf{n}}}} \sqread(T).
%\]
%This completes the proof.
\end{proof}

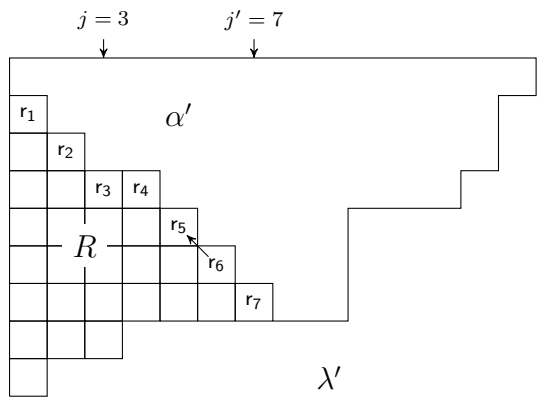
\begin{figure}
\centerfloat
\begin{tikzpicture}[xscale = 0.5,yscale = 0.5,>=stealth']
\tikzstyle{vertex}=[inner sep=3pt, outer sep=0pt]
\tikzstyle{aedge} = [draw, thin, ->,black]
\tikzstyle{edge} = [draw, thick, -,black]
\tikzstyle{dashededge} = [draw, very thick, dashed, black]
\tikzstyle{LabelStyleH} = [text=black, anchor=south, near start]
\tikzstyle{LabelStyleV} = [text=black, anchor=east, near start]

\drawcell{0}{-2}
\drawcell{0}{-1}
\drawcell{1}{-1}
\drawcell{2}{-1}
\drawcell{0}{0}
\drawcell{1}{0}
\drawcell{2}{0}
\drawcell{3}{0}
\drawcell{4}{0}
\drawcell{5}{0}
\drawcell{6}{0}
\drawcell{0}{1}
\drawcell{1}{1}
\drawcell{2}{1}
\drawcell{3}{1}
\drawcell{4}{1}
\drawcell{5}{1}
\drawcell{0}{2}
\drawcell{1}{2}
\drawcell{2}{2}
\drawcell{3}{2}
\drawcell{4}{2}
\drawcell{0}{3}
\drawcell{1}{3}
\drawcell{2}{3}
\drawcell{3}{3}
\drawcell{0}{4}
\drawcell{1}{4}
\drawcell{0}{5}
\draw (7,-0)--(9,0)--(9,3)--(12,3)--(12,4)--(13,4)--(13,6)--(14,6)--(14,7)--(0,7)--(0,6);
%\draw[->] (-.5,7.5) node at (-.5,8) {\tiny$j_0=0$}--(-.5,7);
\draw[->] (2.5,7.5) node at (2.5,8) {\tiny$j=3$}--(2.5,7);
\draw[->] (6.5,7.5) node at (6.5,8) {\tiny$j'=7$}--(6.5,7);
\node[vertex] (a1) at (5.5,1.5) {};
\node[vertex] (a2) at (4.5,2.5) {};
\path[draw] (a1) edge[->,draw] (a2);
\node at (4.5,5.5) {$\alpha'$};
\node at (0.5,5.5) {\tiny$\e{r_1}$};
\node at (1.5,4.5) {\tiny$\e{r_2}$};
\node at (2.5,3.5) {\tiny$\e{r_3}$};
\node at (3.5,3.5) {\tiny$\e{r_4}$};
\node at (4.5,2.5) {\tiny$\e{r_5}$};
\node at (5.5,1.5) {\tiny$\e{r_6}$};
\node at (6.5,0.5) {\tiny$\e{r_7}$};
\node at (8.5,-1.5) {$\lambda'$};
\node[fill=white] at (2,2) {$R$};

%\draw (-1,-3)--(0,-3)--(0,-2)--(1,-2)--(1,-1)--(3,-1)--(3,0)--(7,0)--(7,1)--(6,1)--
%    (6,2)--(5,2)--(5,3)--(4,3)--(4,4)--(2,4)--(2,5)--(1,5)--(1,6)--(0,6)--(0,7)--(-1,7);
%\draw[dashed] (0,-2)--(0,6);
%\draw[dashed] (1,-1)--(1,5);
%\draw[dashed] (2,-1)--(2,4);
%\draw[dashed] (3,0)--(1,5);
%\draw[dashed] (1,-1)--(1,5);

\end{tikzpicture}
\vspace{.3in}
\caption{\label{f main proof} The setup of the proof of Theorem \ref{t flag Schur k3 technical} for
%$\lambda = (14,13,13,12,9,9,9,3,1)$, $\alpha = $
{\footnotesize$\lambda = (9,8,8,7,7,7,7,7,7,4,4,4,3,1)$, $\alpha = (1,2,3,3,4,5,6,7,7,4,4,4,3,1)$}.
A possibility for the arrows of $R$ is shown. A possibility for $\e{w}$ is $\e{r_4r_5r_7}$.}
 \end{figure}

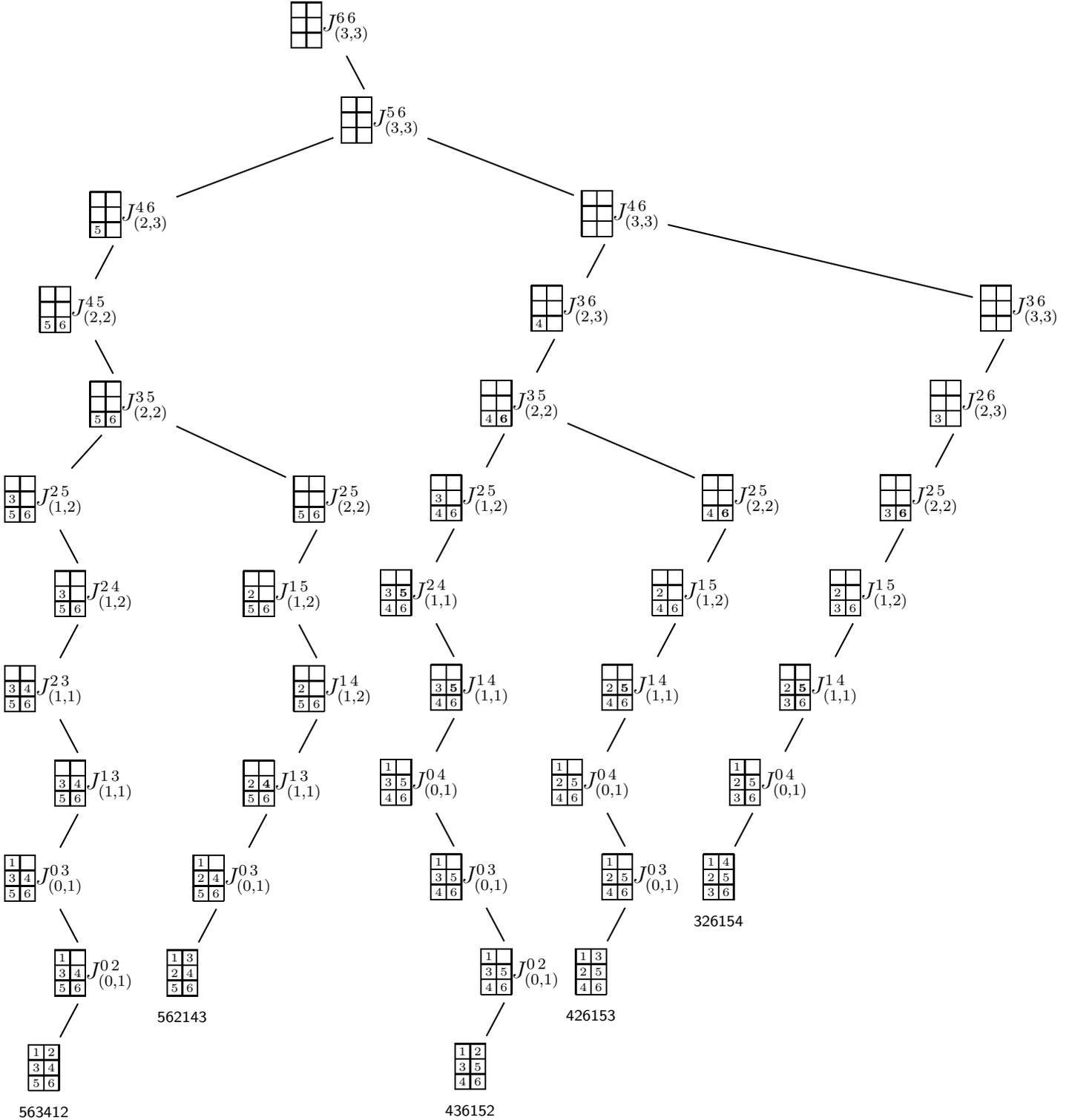
\begin{figure}
\setlength{\cellsize}{1.5ex}
\centerfloat
\begin{tikzpicture}[xscale = 0.9,yscale = 1.7]
\tikzstyle{vertex}=[inner sep=2pt, outer sep=2pt]
\tikzstyle{framedvertex}=[inner sep=3pt, outer sep=2pt, draw=gray]
\tikzstyle{aedge} = [draw, thin, ->,black]
\tikzstyle{edge} = [draw, thick, -,black]
\tikzstyle{doubleedge} = [draw, thick, double distance=1pt, -,black]
\tikzstyle{hiddenedge} = [draw=none, thick, double distance=1pt, -,black]
\tikzstyle{thinedge} = [draw, thin, -,black]
\tikzstyle{dashededge} = [draw, dashed, gray]
\tikzstyle{LabelStyleH} = [text=black, anchor=south]
\tikzstyle{LabelStyleHn} = [text=black, anchor=north]
\tikzstyle{LabelStyleV} = [text=black, anchor=east]
\tikzstyle{LabelStyleSE} = [text=black, anchor=south east]
\tikzstyle{LabelStyleSEn} = [text=black, anchor=north west]
\tikzstyle{LabelStyleSW} = [text=black, anchor=south west]
\tikzstyle{LabelStyleSWn} = [text=black, anchor=north east]

\node[vertex] (v) at (3,-3){${\Tiny\tableau{~&~\\~&~\\~&~}}J_{(3,3)}^{6\,6}$}; %
\node[vertex] (r) at (4,-4){${\Tiny\tableau{~&~\\~&~\\~&~}}J_{(3,3)}^{5\,6}$}; %
\node[vertex] (rl) at (-1,-5){${\Tiny\tableau{~&~\\~&~\\5&~}}J_{(2,3)}^{4\,6}$}; %
\node[vertex] (rll) at (-2,-6){${\Tiny\tableau{~&~\\~&~\\5&6}}J_{(2,2)}^{4\,5}$}; %
\node[vertex] (rllr) at (-1,-7){${\Tiny\tableau{~&~\\~&~\\5&6}}J_{(2,2)}^{3\,5}$};

\begin{scope}[xshift = -20, yshift = 0]
\node[vertex] (rllrl) at (-2,-8){${\Tiny\tableau{~&~\\3&~\\5&6}}J_{(1,2)}^{2\,5}$};
\node[vertex] (rllrlr) at (-1,-9){${\Tiny\tableau{~&~\\3&~\\5&6}}J_{(1,2)}^{2\,4}$};
\node[vertex] (rllrlrl) at (-2,-10){${\Tiny\tableau{~&~\\3&4\\5&6}}J_{(1,1)}^{2\,3}$};
\node[vertex] (rllrlrlr) at (-1,-11){${\Tiny\tableau{~&~\\3&4\\5&6}}J_{(1,1)}^{1\,3}$}; %
\node[vertex] (rllrlrlrl) at (-2,-12){${\Tiny\tableau{1&~\\3&4\\5&6}}J_{(0,1)}^{0\,3}$};
\node[vertex] (rllrlrlrlr) at (-1,-13){${\Tiny\tableau{1&~\\3&4\\5&6}}J_{(0,1)}^{0\,2}$};
\node[vertex,label=below:{\tiny${\e{563412}}$}] (rllrlrlrlrl) at (-2,-14){${\Tiny\tableau{1&2\\3&4\\5&6}}$};
\end{scope}

\begin{scope}[xshift = 30, yshift = 0]
\node[vertex] (rllrr) at (2,-8){${\Tiny\tableau{~&~\\~&~\\5&6}}J_{(2,2)}^{2\,5}$};
\node[vertex] (rllrrl) at (1,-9){${\Tiny\tableau{~&~\\2&~\\5&6}}J_{(1,2)}^{1\,5}$};
\node[vertex] (rllrrlr) at (2,-10){${\Tiny\tableau{~&~\\2&~\\5&6}}J_{(1,2)}^{1\,4}$};
\node[vertex] (rllrrlrl) at (1,-11){${\Tiny\tableau{~&~\\2&\mathbf{4}\\5&6}}J_{(1,1)}^{1\,3}$};
\node[vertex] (rllrrlrll) at (0,-12){${\Tiny\tableau{1&~\\2&4\\5&6}}J_{(0,1)}^{0\,3}$};
\node[vertex,label=below:{\tiny${\e{562143}}$}] (rllrrlrlll) at (-1,-13){${\Tiny\tableau{1&3\\2&4\\5&6}}$};
\end{scope}

\draw[edge] (v) to (r);
\draw[edge] (r) to (rl);
\draw[edge] (rl) to (rll);
%\draw[edge] (rr) to (rrl);
%\draw[edge] (rr) to (rrr);
\draw[edge] (rll) to (rllr);
\draw[edge] (rllr) to (rllrl);
\draw[edge] (rllr) to (rllrr);
\draw[edge] (rllrl) to (rllrlr);
\draw[edge] (rllrlr) to (rllrlrl);
\draw[edge] (rllrlrl) to (rllrlrlr);
\draw[edge] (rllrlrlr) to (rllrlrlrl);
\draw[edge] (rllrlrlrl) to (rllrlrlrlr);
\draw[edge] (rllrlrlrlr) to (rllrlrlrlrl);
\draw[edge] (rllrr) to (rllrrl);
\draw[edge] (rllrrl) to (rllrrlr);
\draw[edge] (rllrrlr) to (rllrrlrl);
\draw[edge] (rllrrlrl) to (rllrrlrll);
\draw[edge] (rllrrlrll) to (rllrrlrlll);

\begin{scope}[xshift = 250, yshift = -85]
\node[vertex] (rr) at (0,-2){${\Tiny\tableau{~&~\\~&~\\~&~}}J_{(3,3)}^{4\,6}$};
\node[vertex] (rrl) at (-1,-3){${\Tiny\tableau{~&~\\~&~\\4&~}}J_{(2,3)}^{3\,6}$};
\node[vertex] (rrll) at (-2,-4){${\Tiny\tableau{~&~\\~&~\\4&\mathbf{6}}}J_{(2,2)}^{3\,5}$};
\node[vertex] (rrlll) at (-3,-5){${\Tiny\tableau{~&~\\3&~\\4&6}}J_{(1,2)}^{2\,5}$};
\node[vertex] (rrllll) at (-4,-6){${\Tiny\tableau{~&~\\3&\mathbf{5}\\4&6}}J_{(1,1)}^{2\,4}$};
\node[vertex] (rrllllr) at (-3,-7){${\Tiny\tableau{~&~\\3&\mathbf{5}\\4&6}}J_{(1,1)}^{1\,4}$};
\node[vertex] (rrllllrl) at (-4,-8){${\Tiny\tableau{1&~\\3&5\\4&6}}J_{(0,1)}^{0\,4}$};
\node[vertex] (rrllllrlr) at (-3,-9){${\Tiny\tableau{1&~\\3&5\\4&6}}J_{(0,1)}^{0\,3}$};
\node[vertex] (rrllllrlrr) at (-2,-10){${\Tiny\tableau{1&~\\3&5\\4&6}}J_{(0,1)}^{0\,2}$};
\node[vertex,label=below:{\tiny${\e{436152}}$}] (rrllllrlrrl) at (-3,-11){${\Tiny\tableau{1&2\\3&5\\4&6}}$};
\end{scope}

\begin{scope}[xshift = 290, yshift = -85]
\node[vertex] (rrllr) at (1,-5){${\Tiny\tableau{~&~\\~&~\\4&\mathbf{6}}}J_{(2,2)}^{2\,5}$};
\node[vertex] (rrllrl) at (0,-6){${\Tiny\tableau{~&~\\2&~\\4&6}}J_{(1,2)}^{1\,5}$};
\node[vertex] (rrllrll) at (-1,-7){${\Tiny\tableau{~&~\\2&\mathbf{5}\\4&6}}J_{(1,1)}^{1\,4}$};
\node[vertex] (rrllrlll) at (-2,-8){${\Tiny\tableau{1&~\\2&5\\4&6}}J_{(0,1)}^{0\,4}$};
\node[vertex] (rrllrlllr) at (-1,-9){${\Tiny\tableau{1&~\\2&5\\4&6}}J_{(0,1)}^{0\,3}$};
\node[vertex,label=below:{\tiny${\e{426153}}$}] (rrllrlllrl) at (-2,-10){${\Tiny\tableau{1&3\\2&5\\4&6}}$};
\end{scope}

\begin{scope}[xshift = 220, yshift = -85]
\node[vertex] (rrr) at (9,-3){${\Tiny\tableau{~&~\\~&~\\~&~}}J_{(3,3)}^{3\,6}$};
\node[vertex] (rrrl) at (8,-4){${\Tiny\tableau{~&~\\~&~\\3&~}}J_{(2,3)}^{2\,6}$};
\node[vertex] (rrrll) at (7,-5){${\Tiny\tableau{~&~\\~&~\\3&\mathbf{6}}}J_{(2,2)}^{2\,5}$};
\node[vertex] (rrrlll) at (6,-6){${\Tiny\tableau{~&~\\2&~\\3&6}}J_{(1,2)}^{1\,5}$};
\node[vertex] (rrrllll) at (5,-7){${\Tiny\tableau{~&~\\2&\mathbf{5}\\3&6}}J_{(1,1)}^{1\,4}$};
\node[vertex] (rrrlllll) at (4,-8){${\Tiny\tableau{1&~\\2&5\\3&6}}J_{(0,1)}^{0\,4}$};
\node[vertex,label=below:{\tiny${\e{326154}}$}] (rrrllllll) at (3,-9){${\Tiny\tableau{1&4\\2&5\\3&6}}$};
\end{scope}

\draw[edge] (r) to (rr);

\draw[edge] (rr) to (rrl);
\draw[edge] (rr) to (rrr);
\draw[edge] (rrl) to (rrll);
\draw[edge] (rrll) to (rrlll);
\draw[edge] (rrll) to (rrllr);
\draw[edge] (rrlll) to (rrllll);
\draw[edge] (rrllll) to (rrllllr);
\draw[edge] (rrllllr) to (rrllllrl);
\draw[edge] (rrllllrl) to (rrllllrlr);
\draw[edge] (rrllllrlr) to (rrllllrlrr);
\draw[edge] (rrllllrlrr) to (rrllllrlrrl);
\draw[edge] (rrllr) to (rrllrl);
\draw[edge] (rrllrl) to (rrllrll);
\draw[edge] (rrllrll) to (rrllrlll);
\draw[edge] (rrllrlll) to (rrllrlllr);
\draw[edge] (rrllrlllr) to (rrllrlllrl);

\draw[edge] (rrr) to (rrrl);
\draw[edge] (rrrl) to (rrrll);
\draw[edge] (rrrll) to (rrrlll);
\draw[edge] (rrrlll) to (rrrllll);
\draw[edge] (rrrllll) to (rrrlllll);
\draw[edge] (rrrlllll) to (rrrllllll);
%\draw[edge] (v) to (l);
%\draw[edge] (v) to (r);
%\draw[edge] (v) to (l);
%\draw[edge] (v) to (r);
\end{tikzpicture}
\caption{\label{f inductive tree}
The inductive computation of $J_{(3,3)}^{6\, 6}$ from the proof of Theorem \ref{t flag Schur k3 technical}, depicted with the following conventions:  instead of writing  $\e{v}J_\alpha^\mathbf{n}(\Jnotb{j}{w})$, we write  $R J_\alpha^\mathbf{n}$ where  $R$ is an RSST with square respecting reading word $\e{vw}$; the entries of  $R$ corresponding to $\e{w}$ are in bold. A left (resp. right) branch of the tree corresponds to the first (resp. second) term of \eqref{e case no square-1}.  All branches that eventually lead to 0 have been pruned.  The words below the leaves indicate the final square respecting reading words produced by this computation.}
\end{figure}

\begin{example}\label{ex inductive tree}
Figure \ref{f inductive tree} illustrates the inductive computation of $J_{(3,3)}^{6\, 6}$ from the proof of Theorem \ref{t flag Schur k3 technical}.

\setlength{\cellsize}{1.9ex}
Write  $A,B,\dots,G$ for  $10,11,\dots, 16$.
Let  $\lambda=(4,4,4,4)$.
Let $R={\tiny\tableau{3 \\ 5 & 6 \\ 7 & 8 & F}}$; then $\sqread(R)=\e{78563F}$.
We illustrate several steps of the inductive computation of $\e{78563F} J_{(1,2,3,4)}^{2 5 E G}$ from the proof of Theorem~\ref{t flag Schur k3 technical}. After each step in which we add an entry to $R$, we record the new values of $R$, $j$, and $j'$.
We first apply the proof of the theorem to $\e{78563F} J_{(1,2,3,4)}^{2 5 E G}$  ($\e{v}=\e{78563F}$,  $\e{w}=\emptyset$,  $R$ as above, $j = 4$,  $j'=3$) and expand as in \eqref{e case no square-1}:
\begin{align}
& \e{78563F} J_{(1,2,3,4)}([2],[5],[E],[G]) \notag \\
=~& \e{78563F} \Big( J_{(1,2,3,3)}([2],[5],[E]   \Jnot{G}   [F]) + J_{(1,2,3,4)}([2],[5],[E],[F]) \, \Big) \label{e ex inductive comp}
\end{align}
Next, apply the theorem to the first term  of \eqref{e ex inductive comp}
($\e{v}= \e{78563F}$,  $\e{w}=\e{G}$, $R={\tiny\tableau{3 \\ 5 & 6 \\ 7 & 8 & F & G}}$, $j=3$, $j'=4$).
This is handled by the case  $w_1 = r_j+1$, hence the computation proceeds by applying \eqref{e case square} and then a  $3$-expansion:
\begin{align*}
& \e{78563F} J_{(1,2,3,3)}([2],[5], [E]  \Jnot{G}   [F]) \\
\equiv ~& \e{78563FG} J_{(1,2,3,3)}([2], [5], [E], [F])\\
=~& \e{78563FG} \Big( J_{(1,2,2,3)}([2], [5]   \Jnot{E}   [D], [F]) + J_{(1,2,3,3)}([2], [5], [D], [F]) \, \Big). \\
\end{align*}
The first term is  $\equiv$ 0 by the proof of Theorem \ref{t flag Schur k3 technical} for the case that (B) holds.  We expand the second term as in \eqref{e case no square-1} ($j=3$):
\begin{align*}
& \e{78563FG} J_{(1,2,3,3)}([2], [5], [D], [F]) \\
=~& \e{78563FG} \Big( J_{(1,2,2,3)}([2],[5]  \Jnot{D}  [C],[F]) + J_{(1,2,3,3)}([2],[5],[C],[F]) \, \Big). \\
\end{align*}

Next, we proceed with the inductive computation of the first term above ($R={\tiny\tableau{3 \\ 5 & 6 & D \\ 7 & 8 & F & G}}$, $j=2$, $j'=4$)
\begin{align*}
& \e{78563FG} J_{(1,2,2,3)}([2],[5]  \Jnot{D}  [C],[F]) \\
=~& \e{78563FG} \Big( J_{(1,1,2,3)}([2]   \Jnot{5}   [4]   \Jnot{D}   [C], [F]) +  J_{(1,2,2,3)}([2], [4]   \Jnot{D}   [C], [F]) \, \Big).
\end{align*}
By Proposition \ref{p squares are strict} and Theorem \ref{t flag Schur k3 technical}, this first term is 0.  We continue the computation with the second term:
%Following \eqref{e case no square0} from the proof, we have
\begin{align*}
& \e{78563FG}  J_{(1,2,2,3)}([2], [4]   \Jnot{D}   [C], [F]) \\
=~& \e{78563FG} \Big( J_{(1,1,2,3)}([2]   \Jnot{4}   [3]   \Jnot{D}   [C], [F]) +  J_{(1,2,2,3)}([2], [3]   \Jnot{D}   [C], [F]) \, \Big) \\
\equiv~& \e{78563FG} \Big( J_{(1,1,2,3)}([2]   \Jnot{4D}   [3] , [C], [F]) +  J_{(1,2,2,3)}([2], [3]   \Jnot{D}   [C], [F]) \, \Big).
\end{align*}
The input data to the theorem for this first term is $\e{v}= \e{78563FG}$,  $\e{w} = \e{4D}$, $R={\tiny\tableau{3 & 4\\ 5 & 6 & D\\ 7 & 8 & F & G}}$, $j=1$, $j'=4$,  $\alpha = (1,1,2,3)$, $\mathbf{n} = (2,3,C,F)$.
%Finally, apply the theorem to this last expression ($j=1$, $j'=4$, $R={\tiny\tableau{3 & 4\\ 5 & {\color{gray}\put(0,1){\vector(-1,1){1}}}\put(-0.25, 0.76){6} & D\\ 7 & {\color{gray}\put(0,1){\vector(-1,1){1}}}\put(-0.25, 0.76){8} & F & G}}$).
%This is handled by the case  $w_1 = r_j+1$, hence by \eqref{e case no square0} followed by a  $1$-expansion,
%\begin{align*}
%& \e{78563FG} J_{(1,1,2,3)}([2]  \Jnot{4D}  [3],[C],[F]) \\
%=~& \e{78563FG4} J_{(1,1,2,3)}([2]  \Jnot{D}  [3],[C],[F]) \\
%=~& \e{78563FG4} \Big(\e{2} J_{(0,1,2,3)}([1]  \Jnot{D}  [3],[C],[F])+  J_{(1,1,2,3)}([1]  \Jnot{D}  [3],[C],[F]) \Big) \\
%\end{align*}
%The first term is 0 by the proof of Theorem \ref{t flag Schur k3 technical} for the case (B) holds.  We continue expanding the second term.
%
%Note that  $\e{78563FG41D}$ is a square respecting reading word of
%\[{\tiny\tableau{1 \\ 3 & {\color{gray}\put(0,1){\vector(-1,1){1}}}\put(-0.25, 0.76){4} \\ 5 & {\color{gray}\put(0,1){\vector(-1,1){1}}}\put(-0.25, 0.76){6} & D \\ 7 & {\color{gray}\put(0,1){\vector(-1,1){1}}}\put(-0.25, 0.76){8} & F & {\color{gray}\put(0,1){\vector(-1,1){1}}}\put(-0.25, 0.76){G}}}.\]
\end{example}

\section{Strengthenings of the main theorem}
\label{s generalizations}
Theorem \ref{t flag Schur k3 technical} does not hold in  $\U/\Ilam{k}$ for $k > 3$, nor does it seem likely that $J_\lambda^\mathbf{n}$ can be computed positively in  $\U/\Ilam{k}$ for  $k>3$ using a sequence of  $j$-expansions and manipulations like that of  Lemma \ref{l lambda1 eq lambda2 commute}.
The papers \cite{BD0graph, BF} discuss the difficulties that must be overcome to apply the Fomin-Greene approach to this more general setting.
Here we discuss two extensions of the main theorem to quotients of  $\U$ that are similar to  $\U/\Ilam{3}$ but have weaker relations.
These relations were inspired by the work of Assaf \cite{Sami, Sami2}.
%, specifically the graphs $\G^{(k)}_{c,D}$ defined in \cite[\textsection4.2]{Sami} to study LLT polynomials (see \cite[\textsection5.3]{BD0graph}).
\subsection{An easy  strengthening}
Define $\U/\Ilam{\le k}$ to be the quotient of  $\U$ by the following relations:
%(let $\Ilam{\le k}$ denote the corresponding two-sided ideal of $\U$):
\begin{alignat*}{2}
\e{v}&= 0 \qquad &&\text{if  $\e{v}\in \U$ such that $\e{v}= 0$ in  $\U/\Ilam{k}$,} \\
u_iu_j &= u_ju_i \qquad &&\text{for $|i-j| > k$,} \\
(u_au_c-u_cu_a)u_b &= u_b(u_au_c-u_cu_a) \qquad &&\text{for  $a < b < c$ and $c-a \le k$.}
\end{alignat*}
All the results in Sections \ref{s reading} and \ref{s positive monomial} hold this variant $\U/\Ilam{\le 3}$ of Lam's algebra  $\U/\Ilam{3}$.
All the proofs in Sections \ref{s reading} and \ref{s positive monomial} work for this variant with no essential change, except that we require \cite{BF} to prove that the elementary symmetric functions commute in  $\U/\Ilam{\le 3}$.  Since $\U/\Ilam{3}$ is a quotient of  $\U/\Ilam{\le 3}$, this is a slight strengthening of the main theorem.

\subsection{A conjectured strengthening}\label{ss a further conjectured strengthening}
Most of our work on this project was towards proving Theorem \ref{t flag Schur k3 technical} in the
algebra $\U/\Irkst{\le 3}$,
where $\U/\Irkst{\le k}$ is the quotient of  $\U$ by the following relations
\begin{alignat}{2}
\e{v}&= 0 \qquad &&\text{if  $\e{v}\in \U$ has a repeated letter,} \label{e repeated letter}\\
u_au_cu_b &= u_cu_au_b \qquad &&\text{for  $a < b < c$ and $c-a > k$,} \label{e rel2} \\
u_b u_au_c &= u_bu_cu_a \qquad &&\text{for  $a < b < c$ and $c-a > k$,} \\
(u_au_c-u_cu_a)u_b &= u_b(u_au_c-u_cu_a) \qquad &&\text{for  $a < b < c$ and $c-a \le k$.} \label{e rel4}
\end{alignat}
The quotient of  $\U$ by \eqref{e rel2}--\eqref{e rel4} has the plactic algebra and $\U/\Ilam{k}$ as quotients.  Also,
the noncommutative elementary symmetric functions commute in this quotient of  $\U$ by \eqref{e rel2}--\eqref{e rel4} and hence in  $\U/\Irkst{\le k}$
(see \cite{BF}).
%; for these and more facts and conjectures about this algebra, see \cite{BF}.
Note that for any  $a < b < c$, the relation in \eqref{e rel4} is implied by either of the following pairs of relations
\begin{align}
u_au_cu_b &= u_cu_au_b \quad \text{ and } \quad u_cu_au_b = u_au_cu_b \quad (\text{Knuth relations})\label{e relplac}\\
u_au_cu_b &= u_bu_au_c \quad \text{ and } \quad u_cu_au_b = u_bu_cu_a \quad (\text{rotation relations}).\label{e relco}
\end{align}
A \emph{bijectivization} of $\U/\Irkst{\le k}$ is a quotient of $\U/\Irkst{\le k}$ obtained by adding, independently for every $a < b < c$ such that $c-a \le k$,
either \eqref{e relplac} or \eqref{e relco} to its list of relations.
\begin{remark}\label{r Sami KR}
Let $\U/\Iassafst{k}$ be the bijectivization of  $\U/\Irkst{\le k}$  which uses only  \eqref{e relco}.
It is related to the graphs $\G^{(k)}_{c,D}$ defined in \cite[\textsection4.2]{Sami} as follows (see \cite[\textsection5.3]{BD0graph} and \cite{BF} for further discussion):
the bijection
$\{(w,c) \in \WRib_k(c,D) \mid \text{$w$ a permutation}\} \to \Wi{k}(c,D')$
of Proposition \ref{p equivalence classes} (iii) is a bijection between the vertex set of  $\G^{(k)}_{c,D}$
and $\Wi{k}(c,D') \subseteq \U$; if each word of $\Wi{k}(c,D')$ contains no repeated letter, then this bijection takes connected components of $\G^{(k)}_{c,D}$ to nonzero equivalence classes of  $\U/\Iassafst{k}$; moreover, all the nonzero equivalence classes of  $\U/\Iassafst{k}$ are obtained in this way.
Here, a \emph{nonzero equivalence class} of an algebra  $\U/I$ is a maximal set  $C$ of words of  $\U$ such that  $\e{v} = \e{w} \not= 0$ in  $\U/I$ for all $\e{v},\e{w} \in C$.
%an algebra if they are equal in $\U/\Iassafst{k}$ and an equivalence class of words is \emph{nonzero} if none of its elements are 0 in  $\U/\Iassafst{k}$.
%;(see \cite[\textsection5.3]{BD0graph}).
\end{remark}

%\begin{lemma}\label{l aaa+1 commute}
%Suppose $\e{m}< x < y$ and  $m\leq n < y$ and  $y-\e{m} > 3$. Then in $\U/\Irkst{\le 3}$,
%\begin{align*}
%J_{(a,a,a+1)}([m]\Jnot{xy}[m], [n]) = \e{x} J_{(a,a,a+1)}([m]\Jnot{y}[m], [n]).
%\end{align*}
%\end{lemma}

We conjecture that Theorem \ref{t flag Schur k3 technical} holds exactly as stated  with  $\U/\Irkst{\le 3}$ in place of  $\U/\Ilam{3}$.
%We briefly outline some of the ingredients we believe are required to prove this conjecture.
%We have a plan of how to prove this with only one piece missing, however one part only works in any bijectivization of $\U/\Irkst{\le 3}$.
We believe that the same inductive structure of the proof works in this setting, except some of the steps are much more difficult to justify.
%We have many of the ingredients of a proof of this conjecture believe that the same general structure of the proof also works in this setting, however some of the
In fact, we know how to prove the slightly weaker statement---that Theorem \ref{t flag Schur k3 technical} holds in any bijectivization of $\U/\Irkst{\le 3}$---assuming the following
\begin{conjecture}\label{cj aaa+1a+2 commute} %???check
Suppose $m \le n_1 \le \cdots \le n_t$,   $n_i < y_i$ for  $i \in [t]$, $m < x$,  $y_1 - x \ge 3$, and  $y_{i+1} - y_i \ge 3$ for  $i \in [t-1]$.
Set  $\alpha = (a,a,a+1,a+2,\ldots, a+t)$.  Then in $\U/\Irkst{\le 3}$,
\begin{align*}
J_{\alpha}([m]  \Jnot{xy_1 \cdots y_t}  [m], [n_1], \ldots, [n_t]) = \e{x} J_{\alpha}([m]\Jnot{y_1\cdots y_t}[m], [n_1], \ldots, [n_t]).
\end{align*}
\end{conjecture}
%Lemma \ref{l lambda1 eq lambda2 commute} and its proof carry over with no essential change to  $\U/\Irkst{\le k}$.
This is a replacement of Corollary \ref{c aaa+1 commute lam} and is needed because, without the far commutation relations, this corollary can no longer be deduced from Lemma \ref{l lambda1 eq lambda2 commute}.
When  $t=0$ this becomes Lemma \ref{l lambda1 eq lambda2 commute} for  $\U/\Irkst{\le 3}$, and the given proof of this lemma carries over with no essential change.
We can also prove the  $t=1$ case, which takes about a page of delicate algebraic manipulations.
Conjecture \ref{cj aaa+1a+2 commute} does not seem to be easy to prove even in the quotient of the plactic algebra by the additional relation \eqref{e repeated letter}.

Another ingredient required to extend the proof of Theorem \ref{t flag Schur k3 technical} to this setting is the following variant of Theorem \ref{t square respecting connected}.
The proof is substantial and requires identifying a nonobvious binomial relation in $\U/\Irkst{\le 3}$.
\begin{theorem}\label{t gen square respecting connected}
Any two square respecting reading words of an RSST are equal in $\U/\Irkst{\le 3}$.
\end{theorem}
%equating two words of length 5 which holds in  that is not a trivial consequence of the given binomial relations.

\section*{Acknowledgments}
I am extremely grateful to Sergey Fomin for his valuable insights and guidance on this project and to John Stembridge for his generous advice and many helpful discussions.  I thank Thomas Lam and Sami Assaf for several valuable discussions and Xun Zhu and Caleb Springer for help typing and typesetting figures.

\bibliographystyle{plain}
\bibliography{mycitations}

\end{document}